\documentclass{article}
\usepackage[utf8]{inputenc}
\usepackage[T1]{fontenc}
\usepackage[english]{babel}
\usepackage{amsfonts,amsmath,amssymb}
\usepackage{mathrsfs}
\usepackage{enumitem}
\usepackage{stmaryrd}
\usepackage{amsthm}
\usepackage{xspace}
\usepackage{color}
\usepackage{geometry}
\usepackage{bbold}  
\usepackage{subfiles}
\usepackage{csquotes} 
\usepackage{biblatex} 

\usepackage[todonotes={textsize=scriptsize}]{changes}
\usepackage[colorlinks=true,linkcolor=blue, citecolor=red]{hyperref}
\usepackage{color}

\usepackage{microtype} 

\newtheorem{Def}{Definition}[section]
\newtheorem{Prop}{Proposition}[section]
\newtheorem{Lemme}[Prop]{Lemma}

\newtheorem*{Rq}{Remark}
\newtheorem{Thm}{Theorem}
\newtheorem{Cor}[Prop]{Corollary}
\newtheorem{Hyp}{Hypotheses}
\newtheorem{Hyp1}[Hyp]{Hypothesis}

\newcommand{\R}{\ensuremath{\mathbb{R} \xspace}}

\newcommand{\T}{\ensuremath{\mathbb{T} \xspace}}
\newcommand{\norm}[1]{\left\lVert #1
  \right\rVert}
\newcommand{\QT}{\ensuremath{{Q_T} \xspace}}
\newcommand{\QTM}{\ensuremath{{Q_{T,M}} \xspace}}
\newcommand{\QtM}{\ensuremath{{Q_{t,M}} \xspace}}
\newcommand{\dd}{\mathrm{d}}
\newcommand{\N}{\ensuremath{\mathcal{N} \xspace}}
\newcommand{\1}{\ensuremath{\mathbb{1} \xspace}}
\newcommand{\E}{\ensuremath{\mathbb{E} \xspace}}
\renewcommand{\P}{\ensuremath{\mathbb{P} \xspace}}

\title{Stability of non-conservative cross diffusion model and approximation by stochastic particle systems}
\author{Vincent Bansaye \footnote{CMAP, INRIA, École polytechnique, Institut Polytechnique de Paris, 91120 Palaiseau, France} \and Alexandre Bertolino \footnote{LJLL (Paris Sorbonne) and  CMAP (Ecole polytechnique)} \and  Ayman Moussa \footnote{Sorbonne Université, CNRS, Université de Paris, Inria, Laboratoire Jacques-Louis Lions (LJLL), F-75005 Paris, France}}
\date{\today}

\begin{document}

\maketitle
\begin{abstract} 
We study the stability 
of  non-conservative deterministic cross diffusion models and prove that they are approximated by stochastic population models when the populations become locally large. In this model,  the   individuals of two species move, reproduce and die with rates sensitive to the local densities of the two species. Quantitative estimates are obtained and convergence is proved soon as  the population per site  and the number of sites go to infinity.
The proofs rely on the extension of stability estimates via duality approach initiated in  \cite{mbh} under a smallness condition and the  development of large deviation estimates for structured population models, which are of  independent interest. The proofs also involve   martingale estimates in $H^{-1}$ and improve the  approximation results in the conservative case as well. 
\end{abstract}

\tableofcontents

\section{Introduction} 


In \cite{SKT}, Shigesada, Kawasaki and Teramoto proposed a system of equations describing the spatial distribution of interacting individuals. This system, named after their initials SKT, focuses on the population densities $u$ and $v$ of two species and include two type of interactions: competition for resources and mutual agitation. These phenomena translate in the model as nonlinearities. Adding a reproduction rate for each species and considering only interspecific agitation, the system takes the general form  
\begin{equation} \label{eq:SKT:origin}
\left\{\begin{array}{ll}
\partial_t u - \Delta( (d_1+a_1 v) u) =  u(\rho_1-s_{11}u-s_{12}v),\\
\partial_t v - \Delta( (d_2+a_2 u) v) = v(\rho_2-s_{21}u-s_{22}v), \\
(u,v)(0, \cdot ) = (u_0, v_0).
\end{array} \right.
\end{equation}
The coefficients $\rho_1$ and $\rho_2$ are the reproduction rates, while the $s_{ij}$ coefficients are the interspecific and intraspecific competition terms for resources. In comparison with standard reaction-diffusion equations, the originality of this system stems from the cross-diffusion terms $a_k$ which model the interspecific agitation. In the previous model all the coefficients are assumed non-negative and $d_1 d_2\neq 0$. This last nondegeneracy condition amounts for the existence, for each species, of intrinsic agitation. This system was originally introduced in \cite{SKT} because the cross-diffusion terms induce a repulsive behavior at the macroscopic level for each species, leading to the existence of (stable) segregated equilibria. The study of these asymptotic states is historically the first mathematical facet that has been explored on the SKT system, and it is still an active domain of research, see for instance \cite{brekuci} and the reference therein. Surprisingly enough, the Cauchy problem for this system only started with \cite{kim1984smooth} and has been puzzling the mathematical community since then. While construction of global weak solutions is now well-understood (\cite{chenjun,dlmt,lepmou,dietmou}) for SKT systems and its variants, their uniqueness and link with more regular solutions for which only local statements are available \cite{amann,chen_jun_unique,chen_jun_unique_bis,galmou,desvmou} offer  challenging questions. In comparison, the derivation of SKT-type models, that is their realization as rigorous limit of microscopic systems in an appropriate scaling limit, is by far a younger area of research as the first rigorous attempt of derivation appeared in 2015 with the program of Fontbona and Méléard \cite{fontmel}. We discard here the fast-reaction limit program proposed in \cite{mimura} and followed for instance rigorously in \cite{DesTres} because this asymptotic may not be understood as a derivation from a simpler (microscopic) model. The approximation of the system by a collection of finite brownian trajectories led to several developments \cite{chendausjunbis,moussa_nonloc_tri} and ultimately to \cite{dietmou}. In an another direction, Daus, Desvillettes and Dietert proposed in \cite{DDD2019} another derivation program which relies on a spatial discretization. In this description, followed in dimension $1$, individuals move (in pairs) on a discrete grid with jumps on neighbor sites which are both random in direction and time. The frequency at which those jumps occur is directly linked to the number of agents present on the site. In   \cite{DDD2019},  the link with the stochastic model was yet formal, while the motion by pair allowed reversibility properties but was not justified in terms of modeling.  An hydrodynamic 
limit has  been implemented in \cite{mbh} leading to a complete derivation of the conservative SKT model in the one dimensional case. However, \cite{mbh} invited several natural generalization and in particular  the case of non-conservative cross-diffusion systems, including reaction terms like in \eqref{eq:SKT:origin} and the relaxation of the scaling required for convergence of the stochastic process. \\

This aim of this paper is to fill this gap. We explore the rigorous derivation of non-conservative SKT models of the following form 
\begin{equation} \label{eq:SKT}
\left\{\begin{array}{ll}
\partial_t u - \Delta( \mu_1(v) u) = R_1(u,v),\\
\partial_t v - \Delta( \mu_2(u) v) = R_2(u,v) \\
(u,v)(0, \cdot ) = (u_0, v_0).
\end{array} \right.
\end{equation}
Here, $R_k$ and $\mu_k$ are generic nonlinearities, with $\mu_1$ and $\mu_2$ lower bounded by a positive constant. As in \cite{mbh}, this SKT system will be obtained as the limit of a collection of interacting random walks, now also branching and dying.  The  new difficulties come from the additional (local) non-linearity due to birth and deaths. The main novelty is the control of the number of birth and deaths  to get  the convergence of the stochastic system and quantify it.  The proofs rely on time  change and comparison arguments and provide  large deviations estimates which are of independent interest. A second input of this work is the extension of the stability estimates of \cite{mbh}, under a smallness condition,  which allows both to cover reaction terms (births and deaths) and more general motion rates.
Last but not least, we improve here the convergence obtained  in \cite{mbh}. More precisely, the convergence of the stochastic process was proven (only) when the number of individuals per site $N$ is large enough compared to the number of sites $M$, namely $N\gg M^2 \gg 1$. Indeed, the normalisation $M^2/N$ appeared naturally in  quadratic variations of martingales when accelerating time by factor $M^2$ and renormalizing population size by $N$. Using here estimates of the jumps of martingales  in $H^{-1}$, we refine the bounds on martingales and only need that  $N,M\rightarrow\infty$. This yields  the minimal setting expected for our convergence and more generally for local approximation of the stochastic process and its deterministic limit. Inspiration here came from the works of Blount \cite{Blount1,Blount2} who considered one single species and linear term for the motion. Our result actually also complement the result of Blount \cite{Blount2} by relaxing scaling  conditions on $N,M$. 
Let us finally  mention that  we refine the distance used to measure the gap between the stochastic process and its deterministic limit, establishing pathwise and uniform-in-time convergence in the spirit of the Skorokhod topology. This  strengthens the approximation result of \cite{mbh}, even in the conservative setting. 
We believe that this approach can be extended to larger dimensions and more complex spatial structure, including random graphs. But this also raises new difficulties, that will be mentioned later, and we hope to achieve it in a forthcoming work. \\ 

In Subsection~\ref{subsec:not} below, we collect all the notations that will be used in this work and then state our main results in Subsection~\ref{subsec:main}. In Section~\ref{sec:2} we state and prove crucial (deterministic) stability estimate both in continuous and semi-discrete setting,  which are then used in Section~\ref{sec:convsto} to establish the convergence of a stochastic process defined on a discrete grid. In  Section~\ref{sec:convsto}, we first construct and characterize the stochastic process. We then prove martingale estimates to control the fluctuations of the stochastic process. In Subsection \ref{sec:nJumps}, we control the number of birth and deaths and the associated cumulative rates along time, but using suitable time change on jump Markov process. We can then combine these ingredients to quantify the convergence of the sequence of stochastic processes to crossed diffusion.

\subsection{Notation and definitions }\label{subsec:not}

We fix $d \in \mathbb N$ an integer. Throughout this article, we denote $\T^d := \R^d / \mathbb Z ^d$ the $d$-dimensional torus, on which for any function $u: \T^d \rightarrow \R$ the usual $L^p$ norm, mean and $L^2$ scalar product are denoted $\norm{u}_{L^p} = \left(\int_{\T^d} |u|^p\right)^{1/p} $, $[u] = \int_{\T^d} u $ and $ \langle u, v \rangle_{L^2}$. We denote $(H^s(\T^d), \norm{\cdot}_{H^s})$ the Sobolev spaces and $\langle \cdot, \cdot \rangle_{H^s}$ the associated scalar product. Their definition and properties are written in Section \ref{sec:sobolev} of the appendix. \\
For a time $T >0$, $Q_T := [0,T] \times \T^d$ is the closed periodic cylinder and we define the following norm:

\begin{equation} \label{eq:defTnorm} |||z |||_{T}^2 = \norm{z}_{L^\infty([0,T],H^{-1}(\T^d))}^2 + \norm{z}_{L^2(\QT)}^2. \end{equation}

For a positive integer $M$, $\T_M$ is the discretized 1-dimensional torus:
$$ \T_M := \{x_j \, |\, j \in \mathbb N \cap [0,M-1] \}, \quad \textnormal{ with } x_j:= \frac j M.$$ 
For a function $\mathbf u \in \R^{\T_M}$, we will often denote $u_j:=\mathbf u(x_j)$. We equip $\T_M$ with the uniform measure of total mass 1 and denote by $\ell^2(\T_M)$  the finite-dimensional vector space of real valued functions over $\T_M$. It will often be equipped with the $\ell^2$ scalar product
$$\langle \mathbf u, \mathbf v \rangle_{2,M} = \frac 1 M \sum_{j=1}^M u_j v_j,$$
sometimes also simply denoted $\langle \mathbf u, \mathbf v \rangle$. {Of course $\mathbf{u}$ and $\mathbf{v}$ can be understood as vectors but throughout all this work, we manipulate these objects as functions and this is why we write $\mathbf{u}\mathbf{v}$ for the element of $\ell^2(\mathbb{T}_M)$ whose $j$-th component is $u_j v_j$.}{ The product of functions $\mathbf u, \mathbf v \in \ell^2(\T_M)$ corresponds to the coordinate by coordinate product and will simply be denoted $\mathbf u \mathbf v$.}\\
For a time $T >0$, $\QTM := [0,T] \times \T_M$ is the discrete analogue of $\QT$, which we equip with the product measure.
For $0 \le i \le M$, the vector $\mathbf e_i \in \ell^2(\T_M)$ is defined by $\mathbf e_M = \mathbf e_0$ if $i=M$ and
$$\mathbf e_i(x_j) = \left\{\begin{array}{ll}
1 \textnormal{ if } i=j\\
0 \textnormal{ else}.
\end{array} \right.$$
We denote by $\mathbf 1 \in \ell^2(\T_M)$ the constant function  on $\T_M$ with value $1$.
For all $\mathbf u \in \ell^2(\T_M)$ we will also use the notation $u_j := \mathbf u(x_{j})$ with the convention $u_M = u_0$. The operator $\Delta_M : \ell^2(\T_M) \longrightarrow \ell^2(\T_M) $ is then defined by
$$ (\Delta_M u)_j = M^2(u_{j+1} + u_{j-1} - 2 u_j).$$
We also define the discrete analogues of $L^p$ norms by setting
$$\norm{\mathbf u}_{p,M} = \left( \frac 1 M \sum_{j=1}^M |u_j|^p \right)^{1/p}, \qquad \norm{\mathbf u}_{\infty} = \max_{1 \le j \le M} |u_j|. $$
We will also denote $[\mathbf u]_M$ the mean of $\mathbf u$:
$$[\mathbf u]_M =\frac 1 M \sum_{j=1}^M u_j. $$
Operator $\Delta_M$ being symmetric with real coefficients and having its kernel equal to constant functions, it induces an automorphism on its range that is the space of all mean free functions. We denote $\Delta_M^{-1}$ the reciprocal automorphism. By integration by parts or Abel transform, one can see that $-\Delta_M$ is non-negative so $-\Delta_M^{-1}$ is positive definite on mean free functions and one can define an analogue of the negative Sobolev norm $H^{-1}$ and the associated scalar product:
$$\langle \mathbf u , \mathbf v \rangle_{-1,M} = [\mathbf u]_M [\mathbf v]_M + \langle \mathbf u -[\mathbf u]_M , -\Delta_M^{-1}(\mathbf v -[\mathbf v]_M) \rangle,$$
$$\norm{\mathbf u}_{-1,M} = \langle \mathbf u , \mathbf u \rangle_{-1,M}^{1/2}.$$
We refer to Subsection \ref{secdiscret} for details.
When nothing is specified on the index, for discrete functions $\langle \cdot, \cdot \rangle$ will denote the $\langle \cdot, \cdot \rangle_{2,M}$ scalar product. For more details about Sobolev spaces and in particular $H^{-1}$, see Appendix \ref{sec:sobolev}. Norm
$||| \cdot |||_{T}$ also has its discrete analogue:
$$ ||| \mathbf u |||_{T,M}^2 = \sup_{t \le T} \norm{\mathbf u(t)}_{-1,M}^2 + \int_0^T \norm{\mathbf u(t)}_{2,M}^2 \, \dd t. $$

We now present interpolations/restrictions allowing  transitions between the continuous and discrete world. For $f : \T \longrightarrow \R$ continuous, we denote by 
$$ \widehat f^M= f_{|\T_M} : \T_M \rightarrow \R$$ the restriction of $f$ to $\T_M$. For $\mathbf u:\T_M \longrightarrow \R$, the $\T \longrightarrow \R$ function $\pi_M(\mathbf u) \in L^2(\T)$ is the affine interpolation of $\mathbf u$:
    $$\pi_M(u)(x) = \sum_{k=1}^M \theta(M(x-x_k))u(x_k), \qquad \textnormal{where }\theta(z)= (1- |z|)^+.$$
To alleviate notations, when no confusion is possible we omit the $M$ exponent. For example, $\widehat u ^M$ becomes $\widehat u$.

For two quantities $A$ and $B$, we will denote $A \lesssim B$ if there exists a constant $c$ depending only on the parameters of the problem, e.g. $\mu_i, R_i$ or $K_0$ and the solution $(u,v)$ defined in Hypotheses \ref{hyp} below, such that $A \le cB$.

For any real number $x$, we denote by $x^+= \max(0, x)$ and $x^-= \max(0, -x)$ respectively its positive and negative parts.

For a matrix $A$, $A^T$ is the transposed matrix and for a probability event $E$, $E^c$ is the complementary event.
For a finite-dimensional vector space $X$ and a time $T>0$, $\mathbb D([0,T], X)$ is the Skorokhod space, i.e. the space of right-continuous functions $f: [0,T] \longrightarrow X$ which have left-hand limits.

Throughout this document, finite dimensional vectors will be noted with bold characters and random variables with capital letters. leading to the following conventions:
\begin{center}
\begin{tabular}{c|c}
        Notation & Object \\ 
        \hline
        $\mathbf u$ & Function on $\T_M$\\ 
        $u$ & Function on $\T$\\ 
        $u_j$ & $j$-th coordinate of $\mathbf u$\\
        $\mathbf U$ & Random function on $\T_M$\\ 
        $U$ & Random real number \\
\end{tabular}
\end{center}
Time-dependent quantities will be indexed by the discretization parameters $M$ and $N$. We choose to omit the corresponding indices in the proofs. For example, a function $\mathbf u^M : \R_+ \rightarrow \ell^2(\T_M)$ and a stochastic process $H^{M,N} : \R_+ \rightarrow \R$ will be respectively denoted $\mathbf u$ and $H$ in proofs. \\
We will also rely on the following setting for the parabolic equations composing the system \eqref{eq:SKT}. 
\begin{Def}
For any $s \in \R$, any bounded function $\mu\in L^\infty(Q_T)$, $F\in L^1([0,T];H^s(\T^d))$ and any $z_0\in L^1(\T^d)$ we say that $z\in L^1(Q_T)$ is a solution of the Cauchy problem
\begin{align*}
    \partial_t z -\Delta(\mu z) &= F,\\
    z(0) &= z^0,
\end{align*}
if, for all $\varphi\in\mathscr{C}_c^\infty([0,T)\times\T^d)$  there holds
\begin{align*}
    -\int_{Q_T} z (\partial_t \varphi + \mu \Delta\varphi) = \int_{\T^d}z^0 \varphi(0) + \int_\QT F\varphi.
\end{align*}

\end{Def}

\subsection{Main results}\label{subsec:main}

This work contains three main results, ranging from the stability of the continuous SKT system to the convergence of the sequence of random processes.

\subsubsection{Stability estimate and semi discrete approximation for the SKT system}

The first result is an extension of the stability estimate given by Theorem 1 of \cite{mbh}. This extension includes reaction terms and more general interspecific agitation (diffusion).  Its proof in Section \ref{sec:2} will also enlighten the other proofs by focusing on a key argument: the use of a duality lemma.
Such approach involves the following smallness conditions  \begin{equation} \label{eq:smallness}
\norm u _{L^\infty(\QT)} \norm v _{L^\infty(\QT)} < \frac {\alpha_1 \alpha_2} {\norm { \mu_1'} _{L^ \infty(\R_+)}\norm {\mu_2'} _{L^\infty(\R_+)}}, 
\end{equation}
where $\norm { \mu_i'} _{L^ \infty(\R_+)}$ stands for the Lipschitz constant of $\mu_i$.
\begin{Thm}\label{th:cont_stab} 
 Assume  that $\mu_i, R_i$ are globally Lipschitz and $\inf_{\R_+}\mu_i = \alpha_i >0$ for $i \in \{ 1, 2\}$. Then for any $T>0$ and  all non-negative solutions $(u,v), (\overline{u}, \overline{v}) \in L^\infty(\QT )$ of \eqref{eq:SKT} with respective initial conditions $(u_0, v_0)$ and $(\overline u_0, \overline v_0)$, and $(u,v)$ verifying the smallness condition \eqref{eq:smallness},
$$ |||u- \overline u |||_T^2 + |||v- \overline v|||_T^2
\lesssim C e^{cT} \left( \norm {u_0 - \overline u_0} _{H^{-1}}^2 +\norm {v_0 - \overline v_0} _{H^{-1}}^2 \right),$$
where $c$ and the constant behind $\lesssim$ depend only on the $L^\infty$ norm and Lipschitz constants of $\mu_i$ and $R_i$ and on the difference between the two members of \eqref{eq:smallness}.
\end{Thm}
\begin{Rq}
Condition \eqref{eq:smallness} is naturally linked to the so-called \textsf{Petrovskii condition} for quasilinear parabolic systems which asks that all eigenvalues of the diffusion matrix have positive real part (we refer to \cite{galmou} for more details on this condition). For the system at stake, namely \eqref{eq:SKT}, the diffusive matrix is given by 
\[\begin{pmatrix}
\mu_1(u_2) & \mu_1'(u_2)u_1 \\
\mu_2'(u_1)u_2 & \mu_2(u_1)
\end{pmatrix}.\]
Since both $\mu_1$ and $\mu_2$ are positively lower-bounded, Petrovskii's condition amounts to ask that the determinant of the previous matrix remains positive that is 
\[\mu_1(u_2)\mu_2(u_1)-\mu_1'(u_2)\mu_2'(u_1)u_1u_2>0.\]
Since $\alpha_i= \inf_{\mathbb{R}_+} \mu_i$, Petrovskii's condition is directly implied by \eqref{eq:smallness}. Note however that Hadamard's ellipticity condition -- which asks that the diffusion matrix has a positive symmetric part and is known to be too restrictive for the study of quasilinear parabolic systems -- is not implied by \eqref{eq:smallness} (again, we refer to \cite{galmou} for a discussion about this other ellipticity condition).
\end{Rq}
As a corollary, we get a uniqueness result for the solutions of \eqref{eq:SKT}:
\begin{Cor} \label{cor:1}
If  $\mu_i$ is uniformly Lipschitz for $i\in\{1,2\}$ and if there exists $T>0$ and  a solution $(u,v) \in L^\infty(\QT)$ of \eqref{eq:SKT} satisfying the smallness condition \eqref{eq:smallness},
then $(u,v)$ is the unique solution in $ L^\infty(\QT)$ of \eqref{eq:SKT} with the initial condition $(u_0,v_0)$.
\end{Cor}


We now turn to the semidiscrete system.
The second result is about the convergence of space discretizations of \eqref{eq:SKT} to a solution of \eqref{eq:SKT}. More precisely, let us introduce a semi-discrete model, first proposed in \cite{daus}. On the discrete torus of dimension 1 $\T_M$, consider the following Cauchy problem:
\begin{equation} \label{eq:discSKT}
\left\{\begin{array}{ll}
 \frac \dd {\dd t} \mathbf u^M = \Delta_M( \mu_1( \mathbf v^M) \mathbf u^M ) + R_1(\mathbf u^M,\mathbf v^M)\\
 \frac \dd {\dd t} \mathbf v^M = \Delta_M( \mu_2( \mathbf u^M) \mathbf v^M ) + R_2(\mathbf u^M ,\mathbf v^M)\\
 (\mathbf u^M(0) , \mathbf v^M(0)) = (\mathbf u^M_0 , \mathbf v^M_0)
\end{array} \right. ,
\end{equation}
where $\mathbf u^M, \mathbf v^M \in \mathcal C ^1( \R_+, \ell^2(\T_M))$ and $\mathbf u^M_0 , \mathbf v^M_0 \in \ell^2(\T_M)$, which is finite dimensional. The study of the stability of this system and its pertubations will a key step between the continuous limiting system (Theorem \ref{th:cont_stab}) and the convergence of the sequence of stochastic processes (Theorem \ref{th:stoc_cv}).
 Throughout the semi-discrete part in Section \ref{sec:2}, we will make the following assumptions:
\begin{Hyp} \label{hyp}
\begin{enumerate}[label=(\roman*)]
    \item We work in dimension $d=1$.
    \item For $i\in\{1,2\}$, functions $\mu_i$ are of class $\mathcal C^3$ and globally Lipschitz and lower-bounded by an $\alpha_i >0$. We denote $L_i$ the Lipschitz constant of $\mu_i$.
    \item For $i\in\{1,2\}$, functions $R_i$    have the structure $R_1(u,v) = u \lambda_1(u,v)$, $R_2(u,v) = v \lambda_2(u,v)$ with $\lambda_1, \lambda_2$ globally Lipschitz and there exists $\rho_0\in \mathbb R$ such that  $$\lambda_1(u,v)+\lambda_2(u,v) \le \rho_0, \quad \lambda_1^-(u,v)^2 + \lambda_2^-(u,v)^2 \lesssim 1+  u(1+ \lambda_1^-(u,v)) + v (1+ \lambda_2^-(u,v)), $$
    where $\lambda_i^-$ is the negative part of $\lambda_i$.
\end{enumerate}
\end{Hyp}
Working in dimension $d=1$ will be technically convenient and particularly useful when comparing the semi-discrete system and the limiting PDE. As commented in the next section, extensions  are expected in higher dimension and even for more general space structures.
\begin{Hyp1} \label{hyp:existSol}
For some $T>0$, there exists a solution $(u,v) \in \mathcal C^0([0,T]; H^2(\T)) \cap L^2([0,T]; H^3(\T))$ of \eqref{eq:SKT} satisfying the smallness condition 
\begin{equation} \label{eq:petitesse2}
\norm u _{L^\infty(\QT)} \norm v _{L^\infty(\QT)} < \frac {\alpha_1 \alpha_2} {\norm { \mu_1'} _{L^ \infty(\R_+)}\norm {\mu_2'} _{L^\infty(\R_+)}}. 
\end{equation} 
\end{Hyp1}
        We note that the solution in Hypothesis \ref{hyp:existSol}
  is unique by Corollary \ref{cor:1} and that the smallness condition makes sense in the class of functions considered since $(u,v) $ is actually bounded thanks to the Sobolev embedding $H^2(\T) \hookrightarrow L^\infty(\T)$.
  Moreover this condition is non-empty and solutions indeed exist. More precisely, in dimension $d$ for $s>d/2$ the Cauchy problem for the SKT is well-posed in the $H^s(\T)$ with local existence (for non-negative initial data) of a non-negative $\mathscr{C}^0([0,T];H^s(\T^d))\cap L^2(0,T;H^{s+1}(\T^d))$ solution \cite{galmou}. 
    Under a smallness assumption on the initial data, this solution is global. We refer to \cite[Theorems 1,2,3]{galmou} for details. \\

We wish to prove that $(\mathbf u^M, \mathbf v^M)$ converges to $(u,v)$. Therefore, we compare $u$ (ou $v$) to the piecewise affine interpolation $\pi_M(\mathbf u^M) \in L^2(\T)$ of $\mathbf u^M$. The second main result quantifies this discretization error:

\begin{Thm}\label{th:disc_stab}
Under  Hypotheses \ref{hyp} and \ref{hyp:existSol},  we 
also assume that
the initial data $\mathbf u_0^M$ and $\mathbf v_0^M$ are non negative and that
 $$  \sup_{M \in \mathbb N} \left( \norm{\mathbf u_0^M}_{2,M}+ \norm{\mathbf v_0^M}_{2,M} \right) < \infty.$$     Then, for any $M\in \mathbb N$, 
$$|||\pi_M(\mathbf u^M) - u |||_T^2+ |||\pi_M(\mathbf v^M) - v |||_T^2 \lesssim  e^{\exp(cT)}(\norm{\pi_M(\mathbf u^M_0) - u_0 }_{H^{-1}(\T)}^2+ \norm{\pi_M(\mathbf v^M_0) - v_0 }_{H^{-1}(\T)}^2 + \delta_M ), $$ 
where $\delta_M\rightarrow 0$ as $M\rightarrow \infty$  and $c$ is a constant. Moreover $(\delta_M)_M$, $c$ and the constant behind $\lesssim$ depend only on $ \sup_{M \in \mathbb N} \left( \norm{\mathbf u_0^M}_{2,M}+ \norm{\mathbf v_0^M}_{2,M} \right)$, the solution $(u,v)$ of \eqref{eq:SKT} and the parameters $\mu_i, R_i$ of the problem. 
\end{Thm}

\begin{Cor}
Under the assumptions of Theorem \ref{th:disc_stab}, if moreover 
$$\norm{\pi_M(\mathbf u^M_0) - u_0 }_{H^{-1}(\T)}^2+ \norm{\pi_M(\mathbf v^M_0) - v_0 }_{H^{-1}(\T)}^2 \underset{M\to\infty}{\longrightarrow} 0,$$
then $(\pi_M(\mathbf u^M), \pi_M(\mathbf v^M) ) \underset{M\to\infty}{\longrightarrow} (u,v)$ for the $|||\cdot |||_T$ norm.  
\end{Cor}

\subsubsection{Convergence of the stochastic particle systems}

First, let us define the stochastic process $(\mathbf U^{M,N}(t) , \mathbf V^{M,N}(t))_{t\geq 0}$ measuring the population densities  along time. The number of sites is still $M\in \mathbb N$.  We also use a scaling of local population size by factor $N\in \mathbb N$.
 Thus, $N U^{M,N}_j(t)$ (resp. $N V^{M,N}_j(t)$) is the number of individuals of the first species  (resp. second species) on site $x_j = j/M$ at time $t$. \\ 
 
 Informally, each individual of species $1$  may move from  site $x_j$ to its left neighbor $x_j-1/M$ and to its right neighbor $x_j+1/M$ with the same individual rate 
 $$M^2\mu_1({v}_j)$$
 where $v_j\in \mathbb R_+$ is the density of individuals of the other species on the same site and $\mu_1$ satisfies the regularity condition of Hypothesis \ref{hyp} $i)$. The random walk associated with each individual is thus centered and scaling $M^2$ is the classical diffusive scaling. \\
 Each individual of species $1$ in each site $j$ may also give birth  on the same site or die with respective individual birth and death rate
 $$b_1({u}_j, {v}_j), \qquad d_1({ u}_j, {v}_j),$$
 where $(u_j,v_j)\in \mathbb R_+^2$ are the local densities of each species.
 Each individual of species $2$ behaves similarly with respective individual rates $M^2\mu_2({u}_j), \,  b_2({u}_j, {v}_j)$ and  $  d_2({u}_j, {v}_j).$ \\
 
 More precisely,
 the transition rates $\nu$ of the process $(\mathbf U^{M,N}(t) , \mathbf V^{M,N}(t))_{t\geq 0}$ are defined as follows for $\epsilon\in \{-1,1\}$ and $1\leq j\leq M$: 
 \begin{align*} 
 ({\bf u}, {\bf v})\,  \longrightarrow \,   & ({\bf u }+ \frac 1 N \mathbf e_{j+\epsilon} - \frac 1 N \mathbf e_j , {\bf v}) \quad \text{at rate} \quad  \nu_1^{M,N}({\bf u}, {\bf v},  \mathbf e_{j+\epsilon} -  \mathbf e_j) :=NM^2u_j\mu_1(v_j)\\
 & ({\bf u }, {\bf v}+ \frac 1 N \mathbf e_{j+\epsilon} - \frac 1 N \mathbf e_j ) \quad \text{at rate} \quad  \nu_2^{M,N}({\bf u}, {\bf v},  \mathbf e_{j+\epsilon} -  \mathbf e_j) :=NM^2v_j\mu_1(u_j)\\
& ({\bf u} + \frac 1 N \mathbf e_{j}, {\bf v})  \quad \text{at rate} \quad \nu_1^{M,N}({\bf u}, {\bf v},  \mathbf e_{j} ) :=Nu_jb_1(u_j,v_j)\\
& ({\bf u}, {\bf v} + \frac 1 N \mathbf e_{j})  \quad \text{at rate} \quad \nu_2^{M,N}({\bf u}, {\bf v},  \mathbf e_{j} ) :=Nv_jb_2(u_j,v_j)\\
& ({\bf u} , {\bf v}- \frac 1 N \mathbf e_{j})   \quad \text{at rate} \quad \nu_2^{M,N}({\bf u}, {\bf v}, - \mathbf e_{j} ) :=Nv_jd_2(u_j,v_j),\\
& ({\bf u} - \frac 1 N \mathbf e_{j}, {\bf v})   \quad \text{at rate} \quad \nu_1^{M,N}({\bf u}, {\bf v}, - \mathbf e_{j} ) :=Nu_jd_1(u_j,v_j).
\end{align*}
We assume that  functions $\mu_i$ and
 $$\lambda_i(u,v) := b_i(u,v) - d_i(u,v)$$ satisfy Hypotheses \ref{hyp} for $i\in\{1,2\}$. 
We give now a trajectorial representation of the process using Poisson point processes.
For convenience, we denote by 
$$\mathcal T = \bigcup_{j=1}^M  \{\mathbf e_j,  -\mathbf e_j, \mathbf e_{j + 1} -\mathbf e_j, \mathbf e_{j - 1} -\mathbf e_j\}$$
the set of the different types of transition, respectively births, deaths, motions to the right and to the left.
Let $\N_1$, $\N_2$ two independent  Poisson random measures on $\R_+ \times \R_+ \times \mathcal T$of intensity $\dd t \otimes \dd \rho \otimes \textnormal{Card}$ where $\textnormal{Card}$ is the counting measure on $\mathcal T$. We consider finite  initial population densities $(\mathbf U^{M,N}(0), \mathbf V^{M,N}(0))$. Then, the process $(\mathbf U^{M,N}, \mathbf V^{M,N})$ 
is the unique strong solution in $\mathbb D(\R^+, \R^{2M}_+)$ of the following SDE
\begin{align} \label{eq:defPS}
    \mathbf U^{M,N} (t) &= \mathbf U^{M,N}_0 + \frac 1 N\int_0^t \int_{\mathbb R_+ \times \mathcal T} 
  \1_{\rho \le \nu_1^{M,N}(\mathbf U^{M,N} (s^-), \mathbf V^{M,N} (s^-), \theta)} \theta    
    \, \N_1( \dd s, \dd \rho, \dd \theta)  \\
    \label{eq:defPS2}
    \mathbf V^{M,N} (t) &= \mathbf V^{M,N}_0 + \frac 1 N\int_0^t \int_{\mathbb R_+ \times \mathcal T} \1_{\rho \le \nu_2^{M,N}(\mathbf U^{M,N} (s^-), \mathbf V^{M,N} (s^-), \theta)} \theta
    \, \N_2( \dd s, \dd \rho, \dd \theta).
\end{align}
We construct rigorously $(\mathbf U^{M,N}, \mathbf V^{M,N})$  and justify this result in Subsection \ref{sec:wp_stoc} using classical localisation arguments and moments estimates on the total population sizes. This result is obtained under Hypothesis \ref{hyp} for coefficients and $L^1$ moment condition for the initial value.  This covers our framework, see forthcoming \eqref{momunM}.   We will observe that construction and pathwise uniqueness, moment conditions can be relaxed by a truncation procedure. Actually, for the proof of the  convergence  we  need higher  moment  assumptions than  bounded  first moments and will assume :
\begin{Hyp} \label{hyp:bd}
\begin{enumerate}
\item[i)] There exist  $p_0>3$ and $K_0>0$ such that for all $K \ge K_0$ and $ M,N \geq 1$,
\begin{equation*} 
 \P\left([\mathbf U^{M,N}_0 ]_M+ [\mathbf V^{M,N}_0]_M \geq K\right) \lesssim \frac{1}{(M K)^{p_0}}.
\end{equation*} 
    \item[ii)] For $i\in \{1,2\}$,  functions $b_i,d_i$ are globally Lipschitz  and there exists $ \rho_0 \ge 0$ and $\alpha < 1 $ such that for any $(u,v)\in \mathbb R_+^2$, $$b_i(u,v) \le \rho_0 + \alpha d_i(u,v), \qquad d_1(u,v)^2 + d_2(u,v)^2 \lesssim 1+  u(1+ d_1(u,v)) + v (1+ d_2(u,v)).$$
\end{enumerate}
\end{Hyp}
Hypothesis \ref{hyp:bd} $i)$ concerns the tail distribution of initial conditions and implies that the $p$-order moments of $[\mathbf U^{M,N}_0 ]_M$ and $[\mathbf V^{M,N}_0 ]_M$ are bounded uniformly in $M,N$ for all $p<p_0$. Let us give two simple and relevant examples, for which the hypothesis is satisfied for any $p_0>3$. This holds  when the initial population densities 
$\mathbf U^{M,N}_0(x_j)$ and $\mathbf V^{M,N}_0(x_j)$ are deterministic and equal to $u_0^{N}(x_j)$ and $v_0^N(x_j)$, where $u_0^N,v_0^N$ are bounded.  The tail condition is also   satisfied when the initial population sizes $N\mathbf U^{M,N}_0(x_j)$ and $N\mathbf V^{M,N}_0(x_j)$ have a Poisson distribution of respective parameters $Nu_0^{N}(x_j)$ and $Nv_0^N(x_j)$, where $u_0^N,v_0^N$ are (still) bounded.\\
Hypothesis \ref{hyp:bd} $ii)$ concerns birth and death rates. It is satisfied for classical reactions terms of the form $$b_1(u,v)=b_1, \quad b_2(u,v)=b_2, \quad  d_1(u,v)=d_1+c_{11}u+c_{12}v,\quad d_2(u,v)=d_2+c_{21}u+c_{22}v,$$ where  $d_i,c_{i,j}$ ($i,j\in \{1,2\})$ are non negative coefficients. This contains the  competitive Lotka Volterra model (logistic reaction, when $c_{ij}>0)$) and  branching processes (linear reaction term, when $c_{ij}=0)$) and the conservative case (no reaction  term, when all coefficients are zero).
It also includes the following  family of birth and death rates : $d_1(u,v) \sim u^\alpha + v^\beta$, $d_2(u,v) \sim u^\gamma + v^\delta$ with $\alpha, \delta \ge 1$, $2\beta \le \delta + 1$ and $2\gamma \le \alpha +1$.
\begin{Thm} \label{th:stoc_cv}
Assume that
 Hypotheses \ref{hyp}, \ref{hyp:existSol} and \ref{hyp:bd} are satisfied.
Then,  there exists a sequence $(\delta_M)_M$ going to $0$ as $M\rightarrow$ and a  constant $c>0$ such that
for any  $N,M\geq 1$, 
\begin{align*}
    &\E \left[|||\pi_M(\mathbf U^{M,N}) - u|||_T^2 +  |||\pi_M(\mathbf V^{M,N}) -v |||_T^2 \right]  \lesssim e^{\exp{(cT)}}\left(D_{0}^{M,N} +\delta_M 
    + \frac {T}{\sqrt{N}} \right),
\end{align*} 
where
\begin{align}
    \label{condinit}
D_{0}^{M,N}&=
\E\left[\norm{\pi_M(\mathbf U^{M,N}_0) - u_0 }_{H^{-1}(\T)}^2+ \norm{\pi_M(\mathbf V^{M,N}_0) - v_0 }_{H^{-1}(\T)}^2\right].
\end{align}
\end{Thm}
This result is stated here for the second moment of our norm $\norm{ \, . \,}$. It   will be  proved more generally for  moments $p\geq 1$. The constant $c$  and the constant behind $\lesssim$
in this estimate do not depend on $T$ and depend on the initial conditions $(\mathbf U^{M,N}_0,\mathbf V^{M,N}_0)$ only through the values of $K_0,\gamma_0$ in Hypothesis \eqref{hyp:bd} $i)$.
We immediately derive from this estimate the following convergence of the stochastic model to the crossed diffusion by letting $M$ and $N$ go to infinity.
\begin{Cor} \label{cor:stoc_cv}
  Assume that
 Hypotheses \ref{hyp}, \ref{hyp:existSol} and \ref{hyp:bd} are satisfied and that 
$D_{0}^{M,N}$ defined in \eqref{condinit} goes to {zero} as $(M,N)$ goes to $(\infty,\infty)$. Then 
Then the following convergence holds:
$$\E\left[ |||\pi_M(\mathbf U^{M,N(M)}) - u|||_T^2 +  |||\pi_M(\mathbf V^{M,N(M)}) -v |||_T^2 \right] \underset{M,N\to\infty}{\longrightarrow} 0.$$
\end{Cor} 
Let us explain  the main ingredients for this convergence and a brief outline of the proof sections. Our approach is pertubative and  relies on stability estimate of the limiting system. Roughly, the semi-martingale decomposition of the stochastic process allows to see it as a deterministic semi-discrete flow with a random source given by a  martingale, which is an integral with respect to compensated Poisson measure. The stability result obtained through  duality estimates in Section \ref{sec:2}
(see Proposition \ref{prop:disc_stab}) allows to prove strong estimates of the form
$$|||\pi_M(\mathbf U^{M,N}) - u|||_T^2 \lesssim e^{c\int_{Q_{T,M}} 1+ (\mathbf U^{M,N})^2+(\mathbf V^{M,N})^2} \, \times \,  \mathcal R^{M,N}(T) \quad \text{a.s.}$$
The process $\mathcal R^{M,N}$ goes to $0$ and quantifies the approximation. It 
 involves the control of martingales for a suitable $H^{-1}$ norm.
The term $\int  1+ (\mathbf U^{M,N})^2+(\mathbf V^{M,N})^2$ is  inherited from  births and deaths via Gr\"onwall inequality. We need fine estimates  to prove that this term is not large and does not degrade the approximation and the convergence  to $0$ coming from $\mathcal R^{M,N}$. These estimates are obtained through large deviation type techniques for multi-type birth and death processes and are
of independent interest.

\section{Stability  of the continuous and semi-discrete systems} \label{sec:2}

\subsection{Duality lemma and proof of Theorem \ref{th:cont_stab}}
To study (\ref{eq:SKT}), it will be useful to look at the following Kolmogorov equation:

\begin{equation} \label{eq:Kolmogorov}
\left\{\begin{array}{ll}
\partial_t z - \Delta( \mu z) = \Delta f + r\\
z(0, \cdot) = z_0
\end{array} \right. .
\end{equation}

To bound the solution of this equation in our context, we rely on a duality lemma. This idea goes back to the works of \cite{Martin1992, pierreSchmitt} on reaction diffusion systems. The lemma we propose to prove is to assimilate with the more recent duality lemma presented in \cite{mbh}. 
This estimate allows for a  precise treatment of $L^2$ source terms, which will be useful to bound the reaction part in \eqref{eq:SKT}.
\begin{Lemme} \label{lem:duality}
    If $\mu \in L^\infty(Q_T)$ is lower-bounded by $\alpha > 0$, then for all $f,r \in L^2(\QT)$ and $z_0 \in H^{-1}(\T^d)$, there exists a unique weak solution $z$ for \eqref{eq:Kolmogorov}. Moreover, $z \in \mathcal C ([0,T], H^{-1}(\T^d))$ and satisfies
    \begin{align*}&\norm { z(T)}_{H^{-1}(\T^d)}^2 + \int_\QT \mu z^2  \\
    &\quad \le \norm {z_0} _{H^{-1}(\T^d)}^2 + \int_0^T  [\mu(s)][z(s)]^2 \, \dd s + 2\int_0^T \norm {r(s)}_{H^{-1}(\T^d)} \norm {z(s)}_{H^{-1}(\T^d)} \dd s  + \frac 1 \alpha \norm f_{L^2(\QT)}^2 .
    \end{align*}
\end{Lemme}
\begin{proof}
When $r$ is mean free, existence and uniqueness come from Lemma 1 of \cite{mbh}, because $r$ can be integrated to $\Delta f$ thanks to Proposition \ref{prop:elliptique} of the appendix. If $r$ has non-zero mean $[r]$, one can consider the solution $w$ to the Cauchy problem
\begin{equation} 
\left\{\begin{array}{ll}
\partial_t w - \Delta( \mu w) = \Delta f + r - [r]\\
z(0, \cdot) = z_0
\end{array} \right. 
\end{equation}
Then, $z(t) := w(t) + \int_0^t [r(s)] \, \dd s$ is the unique solution of \eqref{eq:Kolmogorov}.
We now prove the duality estimate stated in the lemma. We only need to prove it for the case where $\mu, f$ and $r$ are smooth. In this case, it follows from classical parabolic theory (see \cite{evans}, Section 7.1.3, Theorem 7) that $z$ is smooth too. Set $\tilde z(t) := z(t) - [z(t)]$. Let, by Proposition \ref{prop:elliptique} of the appendix, $\phi$ such that for all $t$, $\phi(t)$ is the unique mean free solution of the elliptic equation $\Delta \phi(t) = \tilde z(t)$ on $\T^d$.
We apply the weak formulation of \eqref{eq:Kolmogorov} against $\phi$:
\begin{equation*}
    - \int_\QT z \partial_t \phi    -   \int_\QT z\mu \Delta \phi+ \int_{\T^d} z(T) \phi(T)= \int_\QT (f \Delta \phi + r \phi) + \int_{\T^d} z_0 \phi(0).
\end{equation*}
Replacing $z$ with $[z]+ \Delta \phi$, except in $\int_\QT z\mu \Delta \phi$, and remarking that the integral of $[z]$ against a mean free function is zero, we get
$$- \int_\QT \Delta \phi \partial_t \phi    -   \int_\QT z\mu \tilde z+ \int_{\T^d} \Delta \phi(T) \phi(T)= \int_\QT (f \Delta \phi + r \phi) + \int_{\T^d} \Delta \phi(0) \phi(0).$$
Integrating by parts and noticing $\norm{\nabla \phi(t)}_{L^2(\T^d)}=\norm{\tilde z(t)}_{H^{-1}(\T^d)}$, we then have
$$ \int_0^T \frac 1 2 \frac \dd {\dd t} \norm{\tilde z(t)}_{H^{-1}(\T^d)}^2    -   \int_\QT z\mu \tilde z -\norm{\tilde z(T)}_{H^{-1}(\T^d)}^2= \int_\QT (f \Delta \phi + r \phi) - \norm{\tilde z(0)}_{H^{-1}(\T^d)}^2. $$
Integrating the first term and multiplying by $-2$ on both sides yield
$$\norm{\tilde z(T)}_{H^{-1}(\T^d)}^2 + 2\int_\QT z\mu \tilde z = \norm{\tilde z_0}_{H^{-1}(\T^d)}^2 -2 \int_\QT (f \Delta \phi + r \phi).$$
We now add $2\int_\QT \mu z [z]$ on both sides to complete $2\int_\QT \mu z \tilde z $ in $2\int_\QT \mu z^2 $ and apply Young's inequality to the $f$ term and get, recalling that $\Delta \phi = \tilde z$,
\begin{align*}\norm{\tilde z(T)}_{H^{-1}(\T^d)}^2 + 2\int_\QT \mu z^2 & \le \norm{\tilde z_0}_{H^{-1}(\T^d)}^2 + \int_\QT \left( \frac {f^2} {\mu} + \mu \tilde z ^2 \right) \\
&\qquad + 2 \int_0^T \langle r(s), - \Delta^{-1} \tilde z(s) \rangle_{L^2(\T^d)} \, \dd s+2\int_\QT \mu z [z],
\end{align*}
with
\begin{align*} 
    2\int_\QT \mu z [z] +\int_\QT \mu \tilde z ^2 &=  2\int_\QT \mu z [z] +\int_\QT \mu (z- [z]) ^2 =\int_\QT \mu  [z]^2 + \int_\QT \mu  z^2,
\end{align*}
so
\begin{equation} \label{eq:300}
\norm {\tilde z(T)}_{H^{-1}(\T^d)}^2 + \int_\QT \mu z^2 \le \int_\QT \mu  [z]^2 + \norm {\tilde { z_0}}_{H^{-1}(\T^d)}^2 +\int_\QT \frac {f^2} \mu + 2\int_{0}^T \langle r(s), \tilde z(s)\rangle _{H^{-1}(\T^d)} .    
\end{equation}
To complete $\norm {\tilde z(T)}_{H^{-1}(\T^d)}^2$ into $\norm { z(T)}_{H^{-1}(\T^d)}^2$, there remains to compute $[z(T)]^2$. Integrating \eqref{eq:Kolmogorov} in space one gets
$$\partial_t [z] =[r]. $$
This implies
$$\partial_t( [z]^2) = 2[z]\partial_t [z] = 2[z][r],$$
and then
\begin{equation}\label{eq:15378} [z(t)]^2 =  [z_0]^2 + 2\int_{0}^T [z(s)][r(s)] \, \dd s.\end{equation}
By Pythagorean theorem, $$\norm {\tilde { z}(T)}_{H^{-1}(\T^d)}^2 + [z(T)]^2 = \norm {z(T)}_{H^{-1}(\T^d)}^2 \quad \text{ and } \norm {\tilde { z_0}}_{H^{-1}(\T^d)}^2 + [z_0]^2 = \norm {z_0}_{H^{-1}(\T^d)}^2.$$ Moreover, by orthogonality of mean-free functions with constants in $H^{-1}(\T^d)$,
\begin{multline*}
    \langle r(s), \tilde z(s)\rangle _{H^{-1}(\T^d)} + [r(s)][z(s)] = \langle r(s), \tilde z(s)\rangle _{H^{-1}(\T^d)} + \langle [r(s)], [z(s)]\rangle _{H^{-1}(\T^d)} \\
    = \langle r(s), \tilde z(s)\rangle _{H^{-1}(\T^d)} + \langle r(s), [z(s)]\rangle _{H^{-1}(\T^d)} =  \langle r(s), z(s)\rangle _{H^{-1}(\T^d)} , 
\end{multline*} 
so adding \eqref{eq:15378} with $t=T$ to \eqref{eq:300} we get
$$\norm { z(T)}_{H^{-1}(\T^d)}^2 + \int_\QT \mu z^2 \le \int_\QT \mu  [z]^2 + \norm {z_0}_{H^{-1}(\T^d)}^2 +\int_\QT \frac {f^2} \mu + 2\int_{0}^T \langle r(s),  z(s)\rangle _{H^{-1}(\T^d)} \, \dd s, $$
which, using Cauchy-Schwarz inequality and $ \mu \geq  \alpha$, gives the announced estimate.
\end{proof}

\begin{proof}[Proof of Theorem \ref{th:cont_stab}]
Set $z = u - \overline{u}$, $w = v - \overline{v}$, $L_i = \norm { \mu_i'} _{L^\infty(\mathbb R_+)}$. The strategy is to use the previous lemma to then apply Grönwall's lemma to $z$ and $w$. In order to get closer from the structure of equation (\ref{eq:Kolmogorov}), we remark that $z$ satisfies
$$\partial_t z - \Delta(\mu_1(\overline v) z) = \Delta((\mu_1(v) - \mu_1(\overline v)) u) + R_1(u,v) -  R_1(\overline u ,\overline v).$$
Setting $f= (\mu_1(v) - \mu_1(\overline v)) u $, $\mu= \mu_1(v)$ and $r=R_1(u,v) -  R_1(\overline u ,\overline v)$, Lemma \ref{lem:duality} applies and using $\alpha_1 \le \mu_1$ on its left-hand side we get
\begin{align}
    &\norm { z(T)}_{H^{-1}}^2 + \alpha_1 \int_\QT z^2  \label{firstineg}\\
    &\qquad \qquad \le \norm {z_0} _{H^{-1}}^2 + \int_0^T [\mu_1(\overline v (s))][z(s)]^2 \, \dd s+ 2\int_0^T \norm {r(s)}_{H^{-1}} \norm {z(s)}_{H^{-1}} \dd s + \frac 1 \alpha_1 \norm f_{L^2(\QT)}^2,  \nonumber
\end{align}
where  $H^{-1}$ stands for  $H^{-1}(\T^d)$.
Let us analyse the right-hand side of this inequality, starting with $r(s)$. 
Writing $C_1$ is the Lipschitz constant of $R_1$, $|r(s,x)| \le C_1(|z(s,x)|+|w(s,x)|)$ and
$$\norm {r(s)}_{H^{-1}} \le \norm {r(s)}_{L^2} \le C_1 (\norm {z(s)}_{L^2} +\norm {w(s)}_{L^2}).$$
For the $f$ term, we use Hölder's inequality and the Lipschitz constant $L_1$ of $\mu_1$:
$$\norm {f}_{L^2(\QT)}^2 \le L_1^2 \int_\QT w^2 u^2 \le L_1^2 \norm {u}_{L^\infty(\QT)}^2 \,  \int_\QT w^2.$$
Besides
$[\mu_1(\overline v(s))]$ is upper-bounded by  $\norm {\mu_1} _{L^\infty(\mathbb R_+)}$.
Using these estimates, \eqref{firstineg} becomes
\begin{multline} \label{eq:35}
    \norm { z(T)}_{H^{-1}}^2 + \alpha_1 \int_\QT z^2 \le  \norm {z_0} _{H^{-1}}^2 + \norm {\mu_1} _{L^\infty(\mathbb R_+)}\int_0^T [z(s)]^2 \, \dd s  \\+  2C_1\int_0^T (\norm {z(s)}_{L^2} +\norm {w(s)}_{L^2}) \norm {z(s)}_{H^{-1}} \dd s 
    + \frac {L_1^2} {\alpha_1} \norm {u}_{L^\infty(\QT)}^2 \, \int_\QT w^2.
\end{multline}
We then use Young's inequality $ab \le \frac \varepsilon 2  a^2 + \frac 1 {2\varepsilon} b^2$, with $\varepsilon$ to be determined later, and get:
\begin{align*}
    \int_0^T (\norm {z(s)}_{L^2} +\norm {w(s)}_{L^2}) \norm {z(s)}_{H^{-1}} \dd s
    &\le \frac \varepsilon 2 \int_0^T (\norm {z(s)}_{L^2} +\norm {w(s)}_{L^2})^2 \dd s +  \frac 1 {2\varepsilon} \int_0^T \norm {z(s)}_{H^{-1}}^2 \dd s\\
    &\le \varepsilon  \left(\int_\QT w^2 + \int_\QT z^2\right) + \frac 1 {2\varepsilon} \int_0^T \norm {z(s)}_{H^{-1}}^2 \dd s.
\end{align*}
Dividing the estimate \eqref{eq:35} by $\alpha_1$, we then get
\begin{multline} \label{eq:36}
    \frac 1 {\alpha_1} \norm { z(T)}_{H^{-1}}^2 +\int_\QT z^2 \le  \frac 1 {\alpha_1}\norm {z_0} _{H^{-1}}^2 + \frac{\norm {\mu_1} _{L^\infty(\mathbb R_+)}}{\alpha_1}\int_0^T [z(s)]^2 \, \dd s +   \frac{1} {\alpha_1\varepsilon} \int_0^T \norm {z(s)}_{H^{-1}}^2 \dd s\\
    +\frac {2\varepsilon}{\alpha_1}  \left(\int_\QT w^2 + \int_\QT z^2 \right) + \frac {L_1^2} {\alpha_1^2} \norm {u}_{L^\infty(\QT)}^2\, \int_\QT w^2.
\end{multline}
Passing the $\int_\QT z^2$ term to the left-hand side, then using $[z] \lesssim \norm {z} _{H^{-1}}$ and setting $C_1^\varepsilon = (\norm {\mu_1} _{L^\infty(\mathbb R_+)}+1/\varepsilon)/\alpha_1$,  we obtain
$$\left(1- \frac {2\varepsilon}{\alpha_1}\right) \int_\QT z^2 \le \frac 1 {\alpha_1} \norm {z_0} _{H^{-1}}^2 + C_1^\varepsilon \int_0^T \norm {z(s)}_{H^{-1}}^2 \dd s + \left( \frac {L_1^2} {\alpha_1^2} \norm {u}_{L^\infty(\QT)}^2 + \frac {2\varepsilon}{\alpha_1} \right) \, \int_\QT w^2.$$
We also have a similar inequality on $w$:
$$\left(1-\frac {2\varepsilon}{\alpha_2}\right) \int_\QT w^2 \le \frac 1 {\alpha_1}\norm {w_0} _{H^{-1}}^2 + C_2^\varepsilon \int_0^T \norm {w(s)}_{H^{-1}}^2 \dd s + \left( \frac {L_2^2} {\alpha_2^2} \norm {v}_{L^\infty(\QT)}^2 + \frac {2\varepsilon}{\alpha_2} \right) \, \int_\QT z^2,$$
in which one can inject the one on $\int_\QT z^2$:
\begin{multline*}
    \left(1-\frac {2\varepsilon}{\alpha_2}\right) \int_\QT w^2 
    \le  \frac 1 {\alpha_1}\norm {w_0} _{H^{-1}}^2 + C_2^\varepsilon \int_0^T \norm {w(s)}_{H^{-1}}^2 \dd s\\
    +  \frac{1}{\alpha_1-2\varepsilon} \norm {z_0} _{H^{-1}}^2 + \frac{C_1^\varepsilon}{1-2\epsilon/\alpha_1} \int_0^T \norm {z(s)}_{H^{-1}}^2 \dd s\\
    + \left( \frac {L_2^2} {\alpha_2^2} \norm {v}_{L^\infty(\QT)}^2 + \frac {2\varepsilon}{\alpha_2} \right)\frac 1 {1-\frac {2\varepsilon}{\alpha_1}} \left( \frac {L_1^2} {\alpha_1^2} \norm {u}_{L^\infty(\QT)}^2 + \frac {2\varepsilon}{\alpha_1} \right) \, \int_\QT z^2 .
\end{multline*}
The smallness condition then ensures the existence of $\varepsilon > 0$ such that
$$\left( \frac {L_2^2} {\alpha_2^2} \norm {v}_{L^\infty(\QT)}^2 + \frac {2\varepsilon}{\alpha_2} \right)\frac 1 {1-\frac {2\varepsilon}{\alpha_1}} \left( \frac {L_1^2} {\alpha_1^2} \norm {u}_{L^\infty(\QT)}^2 + \frac {2\varepsilon}{\alpha_1} \right) < 1 - \frac {2\varepsilon}{\alpha_1}.$$
This allows to absorb the $ \int_\QT w^2 $ term of the left-hand side and get:
$$
    \int_\QT w^2 
    \lesssim  \norm {w_0} _{H^{-1}}^2 +\int_0^T \norm {w(s)}_{H^{-1}}^2 \dd s
    +  C\norm {z_0} _{H^{-1}}^2 +\int_0^T \norm {z(s)}_{H^{-1}}^2 \dd s,
$$
where $\varepsilon$ and thus the constant in this inequality depend only on the difference between the two members of \eqref{eq:smallness}.

One can also get a similar estimate on $\int_\QT z^2$. Back to the inequality \eqref{eq:36} with $\norm { z(T)}_{H^{-1}}^2$, we inject the previous bound on $ \norm {w}_{L^2(\QT)}^2 $ to see
\begin{equation*}
    \norm { z(T)}_{H^{-1}}^2+  \norm {z}_{L^2(\QT)}^2
    \lesssim \norm {w_0} _{H^{-1}}^2 + \norm {z_0} _{H^{-1}}^2 +\int_0^T \norm {w(s)}_{H^{-1}}^2 \, \dd s
    +  \int_0^T \norm {z(s)}_{H^{-1}}^2 \, \dd s.
\end{equation*}
Summing this inequality with its analogue for $w$ and setting
\begin{align*}
    \psi(t) &= \norm { z(t)}_{H^{-1}}^2 + \norm { w(t)}_{H^{-1}}^2 ,\qquad
    \theta(t) = \norm {z}_{L^2(Q_t)}^2 + \norm {w}_{L^2(Q_t)}^2,
\end{align*}
we get
$$ \psi(T) + \theta(T) \lesssim \norm {w_0} _{H^{-1}}^2 + \norm {z_0} _{H^{-1}}^2 + \int_0^T \psi(s) \dd s. $$
Grönwall's lemma then gives the wanted result.\end{proof}

\subsection{Space $\ell^2(\T_M)$ and discretization of the continuous solution}
\label{secdiscret}

Let $M \in \mathbb N^*$. In this   subsection, we restrict ourselves to dimension $d=1$ and  study the functional spaces over $\T_M$ and the discrete laplacian operator $\Delta_M$.
We first recall that $\Delta_M$ is a circulant matrix so its eigenvalues are Fourier modes and we can compute explicitly its spectrum, as explained in \cite{circulant} and already used in \cite{mbh}. We then complement by some elementary estimates on discrete norms which will be useful to deal with the jumps of the martingales and the reaction terms.
\begin{Prop} \label{bronespectre}
    The operator $-\Delta_M$ defined on the finite dimensional space $\ell^2(\T_M)$ is symmetric. Its eigenvalues are all contained in $\{0\} \cup [16, 4M^2]$ and $0$ has multiplicity $1$, the kernel of $-\Delta_M$ being composed of constant functions.
\end{Prop}

This proposition implies that $-\Delta_M$ is a positive-definite symmetric operator on the orthogonal of constant functions, i.e. mean free functions. This allows to define $\Delta_M^{-1}\mathbf  u $ for all mean free $\mathbf u$. We can then define the discrete $H^{-1}$ norm

\begin{Def}
The scalar product $\langle u,v\rangle_{-1,M}$ is defined for all $\mathbf u, \mathbf v : \T_M \longrightarrow \R$ by
$$ \langle \mathbf u, \mathbf v\rangle_{-1,M} = [\mathbf u]_M[\mathbf v]_M + \langle \mathbf u - [\mathbf u]_M, -\Delta_M^{-1} (\mathbf v-[\mathbf v]_M)\rangle_{2,M}.$$
Its associated norm is defined by
$$ \norm {\mathbf u}_{-1,M}^2 = \langle \mathbf u,\mathbf u\rangle_{-1,M}. $$
\end{Def}
The fact that $\langle \cdot , \cdot\rangle_{-1,M}$is a scalar product follows immediately from $(-\Delta_M)^{-1}$ being a positive-definite symmetric operator on mean free functions in $\ell^2(\T_M)$.
The spectral gap (independent on $M$, see Proposition \ref{bronespectre}) and the upper bound on the spectrum of $-\Delta_M$ imply the following comparisons between norms:
\begin{Cor} \label{prop:2}
    For $\mathbf u \in \ell^2( \T_M)$, $\norm {\mathbf u}_{-1,M} \le \norm {\mathbf u}_{2,M} \lesssim M \norm {\mathbf u}_{-1,M}$. If $u$ is mean free, then $\norm {\Delta_M^{-1} \mathbf u}_{2,M} \le \norm {\mathbf u}_{2,M}$.
\end{Cor}
The following Proposition is a discrete version of the Sobolev embedding $L^1(\T) \hookrightarrow H^{-1}(\T)$, obtained usually as the dual embedding of $H^{1}(\T) \hookrightarrow L^\infty(\T) $. It is only true in dimension $1$ and is the main use of one-dimensionality in this paper, specially involved in the comparison between the semi-discrete system and the PDE. 
\begin{Prop}
\label{prop:2bis}
For $\mathbf u \in \ell^2( \T_M)$, $\norm{\mathbf u}_{-1,M} \lesssim \norm{\mathbf u}_{1,M}$.
\end{Prop}
\begin{proof}
Denote $\widetilde {\mathbf u}: = \mathbf u - [ \mathbf u]_M$. 
We first prove that $ \norm {\Delta^{-1}_M \widetilde {\mathbf u}}_{\infty} \lesssim \norm { \mathbf u}_{1,M}$. Set
$$\Phi_k = \frac 1 {M^2} \sum_{j=1}^{k} \sum_{i=1}^{j-1}  \widetilde u_i.$$
One can check that $(\Phi_k)_{k \in \mathbb N}$ is $M$-periodic and thus defines an element $\Phi$ of $\ell^2(\T_M)$. Moreover, $(\Delta_M \Phi)_k =  \widetilde u_k$ so $\Phi - [\Phi]_M = \Delta^{-1}_M  \widetilde {\mathbf u}$, hence 
$$\norm {\Delta^{-1}_M  \widetilde {\mathbf u}}_{\infty} \le 2 \norm {\Phi}_{\infty} \le \frac 1 {M^2} \sum_{j=1}^{M} \sum_{i=1}^{M}  |\widetilde u_i| =  \norm {\widetilde{\mathbf u}}_{1,M} \le 2 \norm { \mathbf u}_{1,M},$$
which proves that $ \norm {\Delta^{-1}_M \widetilde {\mathbf u}}_{\infty} \lesssim \norm { \mathbf u}_{1,M}$.
Then,
$$\norm{\mathbf u}_{-1,M}^2  = \langle \tilde {\mathbf u}, -\Delta^{-1}_M  \tilde {\mathbf u} \rangle_{2,M} + [\mathbf u]_M^2 \le \norm{ \tilde {\mathbf u}}_{1,M} \norm{\Delta^{-1}_M \tilde {\mathbf u}}_{\infty}+ \norm{\mathbf u}_{1,M}^2 \lesssim \norm{\mathbf u}_{1,M}^2,$$
which ends the proof.
\end{proof}
\begin{Cor} \label{cor:BoundTM}
Let $\mathbf u : [0,T] \rightarrow \ell^2( \T_M)$. Then,
$$ |||\mathbf u|||_{T,M} \lesssim \sqrt{1+ TM}  \sup_{s \le T} \norm{\mathbf u(s)}_{1,M}.$$
\end{Cor}
\begin{proof}
First, as a consequence of Proposition \ref{prop:2bis} we have
$$\sup_{s \le T} \norm{\mathbf u(s)}_{-1,M}^2 \lesssim \sup_{s \le T} \norm{\mathbf u(s)}_{1,M}^2.$$
Then, bounding $|u_j(s)|$ by $\sum_{j=1}^M |u_j(s)|= M\norm{\mathbf u(s)}_{1,M} $, we get
\begin{align*}
    \int_\QTM \mathbf u^2 &= \frac 1 M \int_0^T \sum_{j=1}^M |u_j(s)|^2 \, \dd s  \\
    &\le \int_0^T \sum_{j=1}^M |u_j(s)| \norm{\mathbf u(s)}_{1,M} \, \dd s =   \int_0^T M \norm{\mathbf u(s)}_{1,M}^2 \, \dd s \le TM  \sup_{s \le T} \norm{\mathbf u(s)}_{1,M}^2.
\end{align*}
Assembling these two bounds yields the result.
\end{proof}
\begin{Prop} \label{prop:jumpSize}
Denoting by $(\mathbf e_i)_{ 1 \le i \le M}$ the canonical basis of $\R^{\T_M}$, the following estimates hold for all $M \in \mathbb N^*$,
$$ \norm{\mathbf e_i}_{-1,M}^2  \lesssim \frac 1 { M} \quad ; \quad \norm{\mathbf e_{i+1} - \mathbf e_{i} }_{-1,M}^2  \lesssim \frac 1 { M^3}.$$
\end{Prop}
\begin{proof}
The first one is a direct consequence of Corollary \ref{prop:2} and the fact that $\norm{\mathbf e_i}_{2,M}^2 = \frac 1 M$. For the second one, note that by translation invariance of the $\norm {\cdot }_{-1,M}$ norm, $\norm{\mathbf e_{i+1} - \mathbf e_{i}}_{-1,M}= \norm{\mathbf e_{1} - \mathbf e_{0}}_{-1,M} $. Moreover, $\mathbf e_{1} - \mathbf e_{0}$ is mean free and 
$$\Delta^{-1}_M (\mathbf e_{1} - \mathbf e_{0}) = \mathbf \Phi, \quad \text{ where } \, \Phi_j = \frac {M+1 - 2j}{2 M^3} \quad  (1 \le j \le M).$$
Then
$$ \norm{\mathbf e_{1} - \mathbf e_{0}}_{-1,M}^2 = \langle\mathbf e_{1} - \mathbf e_{0}, -\Delta^{-1}_M(\mathbf e_{1} - \mathbf e_{0})\rangle_{2,M} = \frac 1 M \left( \frac {M+1 - 2}{2 M^3} - \frac {M+1 - 2M}{2 M^3} \right) \lesssim \frac 1 {M^3},$$
which ends the proof.
\end{proof}
We observe that this computation shows the order of magnitude of respectively birth/death  and motion to neighborhood for the $M$- discretizated $H^{-1}$ norm. In particular,  the normalization $M^3$ of motion event will be able to absorb the acceleration of time $M^2$ involved in the convergence of random walk to brownian motion. It will be useful in following martingale estimate to obtain convergence of the particle systems soon as  $N$ and $M$ go to infinity.\\

We assume the existence of a bounded non-negative solution $(u,v) \in (  L^2([0, T] ; H^3(\T)) \cap L^\infty(Q_T) )^2$ to \eqref{eq:SKT} satisfying the smallness condition \eqref{eq:petitesse2}. Thanks to Corollary \ref{cor:1}, such a solution is unique. Let us first  precise the regularity of $\mu_1(u), \mu_2(v)$:
\begin{Lemme}
 If $u \in  L^2([0, T] ; H^3(\T)) \cap L^\infty(Q_T)$ and $\mu \in C^3(\R)$, then $\mu(u) \in  L^2([0, T ]; H^3(\T))$. 
\end{Lemme}
\begin{proof}
Applying Theorem 2.87 from \cite{jyc}:
$$\norm{\mu(u(t))}_{H^3(\T)} \le \phi_\mu( \norm{u}_{L^\infty(\QT)}) \norm{u(t)}_{H^3(\T)}.$$
Thus, integrating in time, $\norm{\mu(u)}_{ L^2([0, T ]; H^3(\T))} \le \phi_\mu( \norm{u}_{L^\infty(\QT)}) \norm{u}_{ L^2([0, T ]; H^3(\T))} < \infty .$\end{proof}

In order to compare $(u,v)$ with $(\mathbf u^M, \mathbf v^M)$, we begin by putting these two objects in the same space by discretizing $(u,v)$. We thus discretize
$(u,v)$ by restricting it to $\T_M$ and therefore set, for all $t \in [0,T]$, $$\widehat u^M (t) := u_{|\T_M}(t).$$ 
\begin{Prop} \label{prop:rToZero}
    Let $(u,v)$ the solution introduced just above. Then
    \begin{equation*} 
\left\{\begin{array}{ll}
        \frac \dd  {\dd t} \widehat u^M = \Delta_M [\widehat u^M \mu_1 (\widehat v^M)] + R_1(\widehat u^M, \widehat v^M) + \mathbf r_1^M \\
        \frac \dd  {\dd t} \widehat v^M = \Delta_M [\widehat v^M \mu_2 (\widehat u^M)] + R_2(\widehat u^M, \widehat v^M) + \mathbf r_2^M
  \end{array} \right. 
\end{equation*} 
    where the error term ${\mathbf r}$ satisfies ${\mathbf r_1^M} , {\mathbf r_2^M} \underset{M\to\infty}{\longrightarrow} 0$ for the $L^1([0,T]; L^\infty( \T_M))$ norm.
\end{Prop}
\begin{proof}
If $f \in H^3 (\T)$, then using Proposition \ref{prop:FD} from the appendix we get the following convergence in $H^1 (\T)$:
$$ \frac {\tau_h f + \tau_{-h} f - 2f} {h^2}\underset{h\to 0}{\longrightarrow} f'' = \Delta f.$$
By Sobolev embedding, this convergence also holds in $(C(\T) , \norm \cdot_{L^\infty})$ so, applying the restriction application  $f \mapsto \widehat f  $;$$\Delta_M (\widehat f) = M^2(\tau_{\frac 1 M} \widehat f + \tau_{-\frac 1 M} \widehat f- 2\widehat f) = \left(\frac{\tau_{\frac 1 M} f + \tau_{-\frac 1 M} f - 2 f}{1/M^2}\right)_{|\T_M}   \underset{M\to\infty}{\longrightarrow} \widehat {(\Delta f)}.$$
Since $u, \mu_1(v) \in H^3(\T)$ and $H^3(\T)$ is an algebra, one can apply this convergence to $u \mu_1(v) \in H^3(\T)$ and get for almost every $t \in [0,T]$:
\begin{align*}
    \norm{ \mathbf r_1(t)} _ {L^\infty(\T_M)} &=  \norm {\frac \dd  {\dd t} \widehat u(t) - \Delta_M [\widehat u \mu_1 (\widehat v)](t) - R_1(\widehat u, \widehat v)(t) }_ {L^\infty(\T_M)} \\
     &=  \norm {(\partial_t  u  - R_1( u,  v) - \Delta [ u \mu_1 ( v)] + \Delta [ u \mu_1 ( v)])_{|\T_M}(t)  -\Delta_M [\widehat u \mu_1 (\widehat v)](t) }_ {L^\infty(\T_M)} \\
    &= \norm {(\Delta [ u \mu_1 ( v)] )_{|\T_M} (t) -\Delta_M [\widehat u \mu_1 (\widehat v)](t) }_{L^\infty(\T_M)} \underset{M\to\infty}{\longrightarrow} 0 .
\end{align*}
Moreover, using the notation $\Delta_M:= M^2(\tau_{1/ M} + \tau_{- 1/ M} -2)$ on the continuous space $L^2(\T)$, the Sobolev embedding $H^1(\T) \hookrightarrow L^\infty(\T)$ and Proposition \ref{prop:FD} from the appendix gives
\begin{align*}
    &\norm {(\Delta [ u \mu_1 ( v)] )_{|\T_M} (t) -\Delta_M [\widehat u \mu_1 (\widehat v)](t) }_{L^\infty(\T_M)}\\
    &\qquad \qquad=\norm {(\Delta [ u \mu_1 ( v)] -\Delta_M [ u \mu_1 ( v)] )_{|\T_M} (t) }_{L^\infty(\T_M)}\\
    &\qquad \qquad\le \norm {(\Delta [ u \mu_1 ( v)] -\Delta_M [ u \mu_1 ( v)] ) (t)}_{L^\infty(\T)} \\
    &\qquad \qquad\le C \norm {(\Delta [ u \mu_1 ( v)] -\Delta_M [ u \mu_1 ( v)] ) (t)}_{H^1(\T)} \\
    &\qquad  \qquad\le C\norm { u(t) \mu_1 ( v(t))}_{H^3(\T)} \le C\norm { u(t) }_{H^3(\T)} \norm {\mu_1 ( v(t)) }_{H^3(\T)}. 
\end{align*}
Adding that the right hand side belongs to $L^1([0,T])$,   bounded convergence  concludes the proof. 
\end{proof}

\subsection{Stability of pertubated semi-discrete system and proof of Theorem \ref{th:disc_stab}} 

We study now the semi-discrete system $(\mathbf u^M, \mathbf v^M)$ defined in   \eqref{eq:discSKT} and will consider a pertubation. 
First, we control the norm of  the original (non pertubated) system :
\begin{Prop} \label{prop:U1source} 
Under Hypotheses \ref{hyp}, suppose the initial data $(\mathbf u_0^M ,\mathbf v_0^M)$ is non-negative. Then there exists a global non-negative solution $(\mathbf u^M, \mathbf v^M)$ of \eqref{eq:discSKT}, which satisfies for all $T \ge 0$,
$$
    \sup_{t \le T} \norm{\mathbf u^M(t)}_{1,M} + \int_0^T [ \mathbf u^M (s) \lambda^-_1(\mathbf u^M(s),\mathbf v^M(s))]_M \, \dd s \lesssim e^{\rho_0 T}\norm{\mathbf u^M_0}_{1,M} ,$$
$$
    \sup_{t \le T} \norm{\mathbf v^M(t)}_{1,M} + \int_0^T [ \mathbf v^M(s)  \lambda^-_2(\mathbf u^M(s),\mathbf v^M(s))]_M \, \dd s \lesssim e^{\rho_0 T}\norm{\mathbf v^M_0}_{1,M} .$$
\end{Prop}
\begin{proof}
Let $T \in \R_+ \cup \{\infty\}$ the maximal time for which the solution is well-defined. Since $\ell^2(\T_M)$ is finite dimensional, Cauchy-Lipschitz theory proves that $T>0$. \\
Let $C = \{(\mathbf f, \mathbf g) \in \ell^2(\T_M)^2 \, | \,  \mathbf f, \mathbf g \ge 0 \}$ be the closed convex cone of positive functions. We first prove that $(\mathbf u,\mathbf v)$ stays in $C$ until the explosion time $T$. To do so, we prove that for any $t<T$, if $(\mathbf u(t),\mathbf v(t)) \in \partial C$ then $(\mathbf u'(t),\mathbf v'(t))$ either vanishes or points towards the interior of $C$. Let $t< T$ be such that $(\mathbf u(t),\mathbf v(t)) \in \partial C$ and $I = \{1 \le j \le M \, | \, u_j(t) =0 \}$. For all $j \in I$, the structure $R_1(u,v) = u \lambda_1(u,v)$ implies 
$$u_j'(t) = M^2(\mu_1(v_{j-1}(t)) u_{j-1}(t) + \mu_1(v_{j+1}(t)) u_{j+1}(t)) .$$
Since all quantities in the right hand side of this equation are non-negative, we have $u_j'(t) \ge 0$. To prove $(\mathbf u'(t),\mathbf v'(t))$ points towards the interior of $C$, it is therefore sufficient to see that $u_j'(t) > 0$ for at least one $j \in I$. The above expression of $u_j'(t) $ shows that it is the case if $j \in I$ and $j+1 \notin I$ so either such a $j$ exists or $I = \{1, \dots , M\}$ and $\mathbf u'(t)$ vanishes. Repeating this argument on $\mathbf v$ proves that $(\mathbf u,\mathbf v)$ stays in $C$.  

The fact that the explosion time $T$ is infinite will follow from the estimate in the proposition, which we now prove for all $t < T$.
First, one has $\norm{\mathbf u(t)}_{1,M} = \langle \mathbf 1, \mathbf u(t) \rangle_{2,M}$ where $\mathbf 1 \in \ell^2(\T_M)$ is the constant function taking value $1$. Since $\Delta_M \mathbf 1 =0$, for all $\mathbf z \in \ell^2(\T_M)$ one has $\langle \mathbf 1 , \Delta_M \mathbf z \rangle_{2,M} = \langle \Delta_M  \mathbf 1, \mathbf z \rangle_{2,M}=0$. Taking the scalar product of \eqref{eq:discSKTsource} with $\mathbf 1$ and using these facts, one gets 
\begin{align*}
    \norm{\mathbf u(t)}_{1,M} &= \norm{\mathbf u_0}_{1,M}  +  \int_0^t \langle \mathbf 1, \Delta_M( \mu_1( \mathbf v) \mathbf u ) + R_1(\mathbf u,\mathbf v) \rangle_{2,M}(s) \, \dd s \\
    &= \norm{\mathbf u_0}_{1,M}  +  \int_0^t  \langle \mathbf 1,  R_1(\mathbf u(s),\mathbf v(s)) \rangle_{2,M} \, \dd s \\ 
    &=\norm{\mathbf u_0}_{1,M}  +  \int_0^t  \langle \mathbf 1,   \mathbf u(s)\lambda_1(\mathbf u(s),\mathbf v(s))  \rangle_{2,M} \, \dd s.  
\end{align*}
From Hypotheses \ref{hyp} ii),  we get $\lambda_1^+(\mathbf u,\mathbf v) \le \rho_0$ so splitting $\lambda_1 =\lambda_1^+ - \lambda_1^-$ and passing $\lambda^-_1$ to the right-hand side,
$$
    \norm{\mathbf u(t)}_{1,M} + \int_0^t  [\mathbf u  \lambda^-_1(\mathbf u,\mathbf v)](s) \, \dd s \le \norm{\mathbf u_0}_{1,M}  +  \rho_0\int_0^t    \norm{ \mathbf u(s)}_{1,M} \, \dd s $$ 
The Grönwall lemma concludes the proof.
\end{proof}

Now, we study  the semi-discrete system $(\mathbf u^M, \mathbf v^M)$ defined in   \eqref{eq:discSKT} with a perturbation term $\mathbf m^M$:
\begin{equation} \label{eq:discSKTsource}
\left\{\begin{array}{ll}
\mathbf u^M(t) = \mathbf u_0^M +  \int_0^t [\Delta_M( \mu_1( \mathbf v^M) \mathbf u^M ) + R_1(\mathbf u^M,\mathbf v^M)](s) \, \dd s + \mathbf m_1^M (t)\\
\mathbf v^M(t) = \mathbf v_0^M +\int_0^t [\Delta_M( \mu_2( \mathbf u^M) \mathbf v^M ) + R_2(\mathbf u^M ,\mathbf v^M)](s) \, \dd s + \mathbf m_2^M (t)
\end{array} \right.  
\end{equation}
The pertubation term $\mathbf m^M$ will become stochastic in the next section. More precisely, it will be  a martingale coming from compensation of  random jump process due to birth, deaths and motions. This motivates the following class of deterministic pertubations, which mix absolute continuous component and jumps :  
$$ \mathbf m_i^M(t) =   \sum_{k \in \mathbb N} \1_{t_{i,k}^M \le t} \, \mathbf a_{i,k}^M-\int_0^t  \boldsymbol \phi_i^M(s) \, \dd s, $$
where $\boldsymbol \phi_i^M: \R_+ \rightarrow \ell^2(\T_M)$ is a function, $(a_{i,k})_{i\in\{1,2\},k\in \mathbb N}$ is an $\ell^2(\T_M)$-valued sequence and for  $i\in \{1,2\}$, $(t_{i,k})_k$ is an increasing sequence of times which is also divergent :  $t_{i,k} \rightarrow \infty$ as $k\rightarrow \infty$. \\

The following proposition explains how  the pertubated system \eqref{eq:discSKTsource} can be compared to discretization $\widehat u^M (t) := u_{|\T_M}(t)$ of the  continuous limiting system (and  also   to \eqref{eq:discSKT}  recalling Proposition \ref{prop:rToZero}).
\begin{Prop} \label{prop:disc_stab}
   Assume Hypotheses \ref{hyp} and \ref{hyp:existSol} and  let $(\mathbf u^M, \mathbf v^M)$ a non-negative solution of \eqref{eq:discSKTsource} and  set 
   $$\mathbf z^M=\mathbf u^M -\widehat u^M, \qquad \mathbf w^M=\mathbf v^M -\widehat v^M.$$ Then, for any $T\geq 0$ and $M\geq 1$,
    \begin{equation*}|||\mathbf z^M|||_{T,M}^2 + |||\mathbf w^M|||_{T,M}^2 \lesssim e^{\Lambda^M(T)} \, \Big(\norm{ \mathbf z_0^M}_{-1,M}^2 + \norm{\mathbf w_0^M }_{-1,M}^2   + \norm{h_1^M}_{L^\infty_T} + \norm{h_2^M}_{L^\infty_T}+\delta_M\Big),
    \end{equation*}
    where the  constants behind  this inequality are independent of     $T,M$ and   $\mathbf m$, and  
    $$\delta_M= \norm{\mathbf r_1^M}_{L^1([0,T]; L^\infty( \T_M))}^2 + \norm{\mathbf r_2^M}_{L^1([0,T]; L^\infty( \T_M))}^2$$
    goes to $0$ as $M$ tends to infinity from Proposition \ref{prop:rToZero},  and $\Lambda^M$ satisfies
    $$  \Lambda^M(T) \lesssim \int_\QTM 1+ \mathbf u^M +\mathbf v^M +\lambda^-_1( \mathbf u^M, \mathbf v^M)^2 + \lambda^-_2( \mathbf u^M, \mathbf v^M)^2,$$
    and $h^M$ is defined by:
    \begin{align}
    h_1^M(t) &:=   2\sum_{k\in \mathbb N} \1_{t_{1,k} \le t} \, \langle \mathbf z(t_k^- ) , a _{1,k}\rangle_{-1,M} -2 \int^t_0 \langle \mathbf z(s) ,\phi_1(s)\rangle_{-1,M} \, \dd s + \sum_{k\in \mathbb N } \1_{t_{1,k} \le t} \, \norm { a _{1,k}}_{-1,M}^2, \nonumber \\
 h_2^M(t) &:=  2\sum_{k\in \mathbb N} \1_{t_{2,k} \le t} \, \langle \mathbf w(t_k^- ) , a _{2,k}\rangle_{-1,M} -2 \int^t_0 \langle \mathbf w(s) ,\phi_2(s)\rangle_{-1,M} \, \dd s+ \sum_{k\in \mathbb N}  \1_{t_{2,k} \le t} \, \norm { a _{2,k}}_{-1,M}^2 . \nonumber
\end{align} 
\end{Prop}
 Taking $\mathbf m = 0$, this result  implies  that the interpolations $(\pi_M(\mathbf u^M ), \pi_M(\mathbf v^M))$ converge to $(u,v)$ when $M \rightarrow \infty$. This will yield Theorem \ref{th:disc_stab}, see below, after the proof. It will also play a central role in the next section for the convergence of the sequence of stochastic processes by considering the good pertubation term. We recall that in the proofs, we alleviate notation and do not write exponent $M$ in ${\bf z}$, $\Lambda$, $h_i$, etc.
\begin{proof}
As in the continuous case, in order to get this kind of estimate we first remark that $\mathbf z$  satisfies 
$$\mathbf z(t) = \mathbf z_0 + \int_0^t \left( \Delta_M[\mathbf z(s)\mu(s) + f(s)] +x(s)\right) \, \dd s +x_J(t),$$
with $$\mu = \mu_1( \mathbf v), \quad f= \widehat u (\mu_1(\mathbf v) - \mu_1(\widehat v)), \quad x =R_1(\mathbf u,\mathbf v) - R_1(\widehat u, \widehat v) -\phi_1- \mathbf r_1, \quad x_J(t)=\sum_{k \in \mathbb N} \1_{t_{i,k} \le t} \, \mathbf a_{i,k}.$$
This allows to apply the duality lemma to set up for a Grönwall lemma. Therefore, we can invoke \cite{mbh} and more precisely  we use $\mu_1 \ge \alpha_1$ and Lemma \ref{lem:discont} in Appendix  to  get
\begin{multline} \label{eq:discdual}
    \norm {\mathbf z(t)} ^2_{-1,M} + \alpha_1 \int_\QtM \mathbf z^2
    \le \norm {z_0} ^2_{-1,M} + \int_0^t [\mathbf z (s)]_M^2[\mu(s)]_M \, \dd s + h_1(t)\\
    +\frac 1 {\alpha_1} \int_\QtM f^2 + 2 \int^t_0 \langle \mathbf z(s) ,y(s) \rangle_{-1,M} \, \dd s,
\end{multline}
where $h_1(t)$ is as defined in the above statement and
$$y=R_1(\mathbf u,\mathbf v) - R_1(\widehat u, \widehat v) - \mathbf r_1= x+\phi_1.$$
Let us inspect bounds for the terms on the right-hand side in their order of appearance. 

First, for the $[z]_M^2$ term, we use $[z(s)]_M^2 \le \norm {z(s)} ^2_{-1,M}$, which is a direct consequence of the definition of $\norm {.} ^2_{-1,M}$.
Then, since $\mu_1$ is Lipschitz-continuous,
$$\int_\QtM f^2 \le L_1^2 \int_\QtM  \widehat u^2 |\mathbf v - \widehat v |^2 \le  L_1^2  \norm u _{L^\infty(\QT)}^2 \int_\QtM \mathbf w^2.  $$ 
For the last term, since $\norm {\cdot}_{-1,M} \le \norm {\cdot}_{2,M} \le \norm {\cdot}_{\infty}$ and $\norm {\cdot}_{-1,M} \lesssim \norm {\cdot}_{1,M}$ from Corollary \ref{prop:2} and Proposition \ref{prop:2bis},
\begin{align*}
    \langle \mathbf z(s) ,y(s)\rangle_{-1,M} & \le \norm {\mathbf z(s)}_{-1,M} \norm {y(s)}_{-1,M} \\
    &\le  \norm {\mathbf z(s)}_{-1,M} \norm {   R_1(\mathbf u,\mathbf v)(s) -R_1(\widehat u, \widehat v)(s)}_{-1,M} + \norm {\mathbf z(s)}_{-1,M}\norm {\mathbf r_1(s)}_{-1,M} \\
    & \lesssim \norm {\mathbf z(s)}_{-1,M} \norm { R_1(\mathbf u,\mathbf v)(s) - R_1(\widehat u, \widehat v)(s)  }_{1,M} + \norm {\mathbf z(s)}_{-1,M}\norm {\mathbf r_1(s)}_{\infty},
\end{align*}
with, since $\lambda_1$ is Lipschitz-continuous,
\begin{align*}
    |R_1(\mathbf u,\mathbf v)-R_1(\widehat u, \widehat v)| &= |\mathbf u \lambda_1(\mathbf u,\mathbf v)-\widehat u \lambda_1(\widehat u, \widehat v)| \\
    &\le  | \mathbf u \lambda_1(\mathbf u,\mathbf v)-\widehat u \lambda_1(\mathbf u, \mathbf v) | + |\widehat u \lambda_1(\mathbf u, \mathbf v) - \widehat u \lambda_1(\widehat u, \widehat v) | \\ 
    &\lesssim |\mathbf z \lambda_1(\mathbf u,\mathbf v) | +\widehat u(|\mathbf z| + |\mathbf w|).
\end{align*}
It implies, since $\widehat u$ is smaller than $\norm{u}_\infty$ which is a constant,
\begin{align*}
    &\norm {  R_1(\mathbf u(s),\mathbf v(s)) - R_1(\widehat u(s), \widehat v(s))}_{1,M} \\
    &\hspace{4cm} \lesssim \norm{\lambda_1( \mathbf u(s), \mathbf v(s))}_{2,M}\norm{\mathbf z(s)}_{2,M} +  \norm{\widehat u}_{2,M}(\norm{\mathbf z(s)}_{2,M} + \norm{\mathbf w(s)}_{2,M}) \\
    &\hspace{4cm} \lesssim (1+\norm{\lambda_1( \mathbf u(s), \mathbf v(s))}_{2,M})(\norm{\mathbf z(s)}_{2,M} + \norm{\mathbf w(s)}_{2,M}).
\end{align*}
Thus, using Young's inequality with $\varepsilon >0$ to choose later,
\begin{align*}
    & \langle\mathbf z(s) ,x(s)\rangle_{-1,M}\\
    &\qquad  \lesssim \norm {\mathbf z(s)}_{-1,M} (1+\norm{\lambda_1( \mathbf u(s), \mathbf v(s))}_{2,M})  (\norm{\mathbf z(s)}_{2,M} + \norm{\mathbf w(s)}_{2,M}) +\norm{\mathbf z(s)}_{-1,M}\norm {\mathbf r_1(s)}_{\infty}\\
    &\qquad \lesssim \varepsilon^{-1} ( 1 + \norm{\lambda_1( \mathbf u(s), \mathbf v(s))}_{2,M}^2)\norm {\mathbf z(s)}_{-1,M}^2 + \varepsilon (\norm{\mathbf z(s)}_{2,M}^2 + \norm{\mathbf w(s)}_{2,M}^2) +  \norm {\mathbf z(s)}_{-1,M}\norm {\mathbf r_1(s)}_{\infty} .
\end{align*}

Plugging all these bounds in \eqref{eq:discdual} and gathering terms involving $\norm {\mathbf z(s)}_{-1,M}$, we get
\begin{align}\label{eq:applique_lemme_4}
  &  \norm {\mathbf z(t)} ^2_{-1,M} + \alpha_1\int_{\QtM} \mathbf z^2  \\
  &\qquad \qquad \lesssim \frac {L_{1}^2} {\alpha_1} \norm u _{L^\infty(Q_T)}^2 \int_\QtM \mathbf w^2
    + \varepsilon  \int_\QtM (\mathbf z^2 +\mathbf w^2)
    +g_{1}(t) + h_1(t),\nonumber
\end{align} 
where 
\begin{align}
g_{1}(t)&=  \norm {\mathbf z(0)} ^2_{-1,M} + \int_0^t \norm {\mathbf z(s)} ^2_{-1,M}[\mu_1(\mathbf v(s))]_M \, \dd s  \label{formg1} \\
& \qquad + \int_0^t \left(\varepsilon^{-1}(1 + \norm{\lambda_1( \mathbf u(s), \mathbf v(s))}_{2,M}^2)  \norm {\mathbf z(s)}_{-1,M}^2 + \norm {\mathbf z(s)}_{-1,M} \norm {\mathbf r_1(s)}_{\infty} \right) \,\dd s.  \nonumber
\end{align}
In particular, we have
\begin{equation} \label{ineq:3}
    (\alpha_1 - \varepsilon) \int_{\QtM} \mathbf z^2
    \lesssim \left( \frac {L_{1}^2} {\alpha_1} \norm u _{L^\infty(Q_T)}^2 +\varepsilon \right)\int_{\QtM} \mathbf w^2
    +g_{1}(t) + h_1(t).
\end{equation}
There is a similar inequality for $\mathbf w$:
\begin{equation*}
    (\alpha_2 - \varepsilon) \int_{\QtM} \mathbf w^2
    \lesssim \left( \frac {L_{2}^2} {\alpha_2} \norm v _{L^\infty(Q_T)}^2 +\varepsilon \right) \int_{\QtM} \mathbf z^2
    +g_{2}(t)  + h_2(t).
\end{equation*}
Injecting \eqref{ineq:3} in it to bound $\int_{\QtM} \mathbf z^2$, one gets
\begin{align*}
    (\alpha_2 - \varepsilon) \int_{\QtM} \mathbf w^2
    &\lesssim A_{\varepsilon} \int_{\QtM} \mathbf w^2  +  h_1(t) + g_{2}(t) +  h_2(t),
\end{align*}
where
$$A_{\varepsilon}=\left( \frac {L_{2}^2} {\alpha_2} \norm v _{L^\infty(Q_T)}^2 +\varepsilon \right)  \frac 1 {\alpha_1 - \varepsilon}  \left( \frac {L_{1}^2} {\alpha_1} \norm u _{L^\infty(Q_T)}^2 +\varepsilon \right).$$
By the smallness Assumption  \eqref{eq:petitesse2},
for $\varepsilon$ small enough,
$A_{\varepsilon} <\alpha_2 - \varepsilon$. 
Therefore term $\int_{\QtM} \mathbf w^2$  from the right-hand side can be sent to the left-hand side to get
$$
 \int_{\QtM} \mathbf w^2
    \lesssim 
    g_{2}(t) + g_{1}(t) + h_1(t) + h_2(t).
$$
The choice of $\varepsilon$ depends only on the parameters of the problem and $(u,v)$ (through the dependence on the difference of the two members of the smallness condition), so the constant behind $\lesssim$ depends only on these same objects.
A similar inequality holds for $\int_{\QtM} \mathbf z^2$.  Back to \eqref{eq:applique_lemme_4}, we get by summing it with its analogue for $\mathbf w$ and using the previous bounds on $\int_{\QtM} \mathbf z^2$ and $\int_{\QtM} \mathbf w^2$ that
\begin{equation} \label{eq:400}
    \norm {\mathbf z(t)} ^2_{-1,M} + \int_{\QtM} \mathbf z^2 +\norm {\mathbf w(t)} ^2_{-1,M} + \int_{\QtM} \mathbf w^2
    \lesssim g_{1}(t)  +g_{2}(t) + h_1(t) + h_2(t).
\end{equation}
Recalling from
\eqref{formg1}
the expression of $g_1$ and that $g_2$ has a similar form and setting 
$$\psi(t) := \norm {\mathbf z(t)} ^2_{-1,M} + \norm {\mathbf w(t)} ^2_{-1,M},\qquad  \phi(t) :=  \int_{\QtM} \mu_1(\mathbf v) \mathbf z^2 + \int_{Q_t} \mu_2(\mathbf u) \mathbf w^2,$$
$$a(t) := 1 + [\mu_2(\mathbf u(t))]_M + [\mu_1(\mathbf v(t))]_M + \norm{\lambda_1( \mathbf u(s), \mathbf v(s))}_{2,M}^2 + \norm{\lambda_2( \mathbf u(s), \mathbf v(s))}_{2,M}^2,$$
inequality
 \eqref{eq:400}  rewrites
$$ \psi(t) + \phi(t) \lesssim \psi(0) + \int_0^t [a(s)\psi(s) + (\norm {\mathbf r_1(s)}_{\infty} +\norm {\mathbf r_2(s)}_{\infty}) \psi(s)^{1/2}  ] \dd s +  h_1(t) + h_2(t). $$
Bounding $h_1(t) + h_2(t) $ by $\norm{h_1}_{L^\infty_T} + \norm{h_2}_{L^\infty_T}$ and using non-linear Grönwall lemma (see Lemma \ref{lem:Gronwall}), we get
\begin{align*} &\sup_{t \le T} \psi(t)  \, + \, \phi(T) \}  \\
&\qquad \lesssim e^{c\int_0^T a(s) \, \dd s} \left( \psi(0)+ \norm{h_1}_{L^\infty_T} + \norm{h_2}_{L^\infty_T} +  \norm {\mathbf r_1}_{L^1([0,T];L^\infty(\T_M)}^2 + \norm {\mathbf r_2}_{L^1([0,T];L^\infty(\T_M)}^2  \right).
\end{align*}
It remains to control $a$ and conclude.
Using Hypotheses \ref{hyp}, we   get $\mu_i(x) \lesssim 1 +x$ and $|\lambda_i(u, v)| \le |\lambda^+_i( u, v)| + |\lambda^-_i( u, v)| \le \rho_0 + |\lambda^-_i( u, v)| $,
\begin{align*}
    \int_0^T a(s) \, \dd s &\lesssim \int_0^T  1 + [\mathbf u(s)]_M + [\mathbf v(s)]_M + \norm{\lambda^-_1( \mathbf u(s), \mathbf v(s))}_{2,M}^2 + \norm{\lambda^-_2( \mathbf u(s), \mathbf v(s))}_{2,M}^2 \, \dd s \\
    &\lesssim    \int_\QTM (1+ \mathbf u +\mathbf v +\lambda^-_1( \mathbf u, \mathbf v)^2 + \lambda^-_2( \mathbf u, \mathbf v)^2 ).
\end{align*}
It ends the proof.
\end{proof}

We are now in position to prove Theorem \ref{th:disc_stab} stated in the first section.
\begin{proof}[Proof of Theorem \ref{th:disc_stab}]
 In the case of the semi-discrete system with no perturbations \eqref{eq:discSKT}, $\mathbf m =0$ so  $h_1$ and $h_2$ from Proposition \ref{prop:disc_stab} vanish and Hypothesis \ref{hyp} iii) and Proposition \ref{prop:U1source} yield
\begin{align*}
     \Lambda(T) &=\int_\QTM 1+ \mathbf u +\mathbf v +\lambda^-_1( \mathbf u, \mathbf v)^2 + \lambda^-_2( \mathbf u, \mathbf v)^2 \\
     &\lesssim \int_\QTM 1+ \mathbf u +\mathbf v +\mathbf u \lambda^-_1( \mathbf u, \mathbf v) + \mathbf v\lambda^-_2( \mathbf u, \mathbf v) \lesssim e^{c T}(\norm{\mathbf u_0}_{1,M} +\norm{\mathbf v_0}_{1,M}) \lesssim \exp(cT).
\end{align*}
Thus, the estimate from Proposition \ref{prop:disc_stab} rewrites
$$|||\mathbf z|||_{T,M}^2 + |||\mathbf w|||_{T,M}^2 \lesssim e^{\exp(cT)} (\norm{ \mathbf z_0}_{-1,M}^2 + \norm{\mathbf w_0 }_{-1,M}^2+ \delta_M),$$
where $\delta_M$   converges to $0$ according to Proposition \ref{prop:rToZero}. Note that the constant in front of $\exp(cT)$ has
been by changing the value of $c$.
Then, forthcoming Proposition \ref{prop:interpol} in Appendix ensures that
$$ |||\pi_M(\mathbf u) - u |||_T \lesssim |||\mathbf z|||_{T,M} + \delta_M',$$
 $$ \norm{ \mathbf z_0}_{-1,M} \lesssim \norm{\pi_M(\mathbf u_0) - u_0 }_{H^{-1}(\T)} + \delta_M' ,$$
 where $\delta_M'$ goes to $0$ and  can be absorbed in $\delta_M$. Combining  the  previous inequalities ends  the proof.
\end{proof}.

\section{Convergence of stochastic processes}\label{sec:convsto}

In this chapter, we consider an individual based model, living on the discrete space $\mathbb T_M$. We prove that when the local population scale $N$ goes to infinity,   the stochastic process is close to   the semi discrete system. We obtain  quantitative estimates which are   uniform  with respect to the number of sites $M$. It will prove that  when $M$ and $N$ become large, the individual based model converges to the crossed diffusion system.  The   norm is inherited from the duality estimates. It relies also on the control of fluctuations  through martingales estimates in $H^{-1}$ and large deviations estimates for births and deaths terms. We first justify the existence of stochastic processes and provide preliminary moment estimates. 
We then derive the martingales inequalities (Subsection \ref{Secmartingales}) and study the large deviations (tail estimates) in Subsection 
\ref{sec:nJumps}. We will then combine the results to prove convergence of the stochastic process  in Subsection \ref{secpreuveth3} (Theorem \ref{th:stoc_cv}).
\subsection{Construction of the process and first estimates} \label{sec:wp_stoc}
We consider $p\geq 1$ and assume here that initial conditions satisfy
\begin{equation}
    \label{momunM}
    \E \left[ \norm{ \mathbf U^{M,N}(0)}_{1,M}^p + \norm{ \mathbf V^{M,N}(0)}_{1,M}^p \right] <\infty.
\end{equation}
It  will cover our framework. Indeed, Hypothesis \ref{hyp:bd} $i)$ involved for convergence will ensure that \eqref{momunM} is satisfied for   $p < p_0$.  \\ 

For the proofs, we proceed by localization of processes to keep them bounded and work on a finite state space. More precisely, 
we first define  the  c\`adl\`ag jump Markov process $(\mathbf U^{M,N}, \mathbf V^{M,N})$  satisfying   $\eqref{eq:defPS},\eqref{eq:defPS2}$ until time
$$T^K=\inf\{ t\geq 0 : \norm{\mathbf U^{M,N}(t)}_\infty + \norm{\mathbf V^{M,N}(t)}_\infty \geq K\} \in [0,\infty].$$
Indeed, the process $(\mathbf U^{M,N}, \mathbf V^{M,N})$ remains on  finite state space 
$ [0,K+1]^{2M}$
in the time interval $[0,T^K] \cap \R_+$.
 The process is now simply a jump Markov process on a finite state space and its construction and characterization are classical. Thus, it can be   constructed by iteration, considering the successive jumps, which  provides   a   strong solution   to  $\eqref{eq:defPS}-\eqref{eq:defPS2}$  on the time interval  $[0,T^K] \cap \R_+$. Let us also note that pathwise uniqueness on this time interval immediately holds.\\

 Let us now prove that $T^K$ goes to infinity as $K$ goes to infinity, using the  the control  the growth rate (Hypothesis \ref{hyp} $iii)$) and our moment condition above.
First, we observe that compensating the Poisson random measure in  $\eqref{eq:defPS}$ gives the following  semi-martingale decomposition (see Section II-3 \cite{ikeda}):
\begin{equation}
  \label{decsemU} 
\mathbf U(t) = \mathbf A(t) +  \mathcal M_1(t),
\end{equation}
for $t\in \mathbb R_+\cap T^K$, 
where
\begin{align*}
    \mathbf A^{M,N}(t) &= \mathbf U_0^{M,N} + \frac 1 N \int_0^{t} \int_{\mathbb R_+ \times \mathcal T}  \1_{\rho \le \nu_1(\mathbf U^{M,N}(s^-), \mathbf V^{M,N}(s^-), \theta)} \theta \,  \dd s \dd \rho \dd \theta \\
   &=\mathbf U_0^{M,N} + \int_0^{t} \phi_1(\mathbf U^{M,N}(s), \mathbf V^{M,N}(s)) \dd s   
\end{align*}
and 
\begin{equation}
\label{defphi1}
\phi_1( \mathbf u, \mathbf v) =  \sum_{\theta \in \mathcal T}\frac 1 N  \nu_1(\mathbf u, \mathbf v , \theta) \theta =
    \left[ \Delta_M(\mu_1(\mathbf v)\mathbf u) + R_1(\mathbf u, \mathbf v)\right]
    \end{equation}
and
 $(\mathcal M_1(t\wedge T^K))_{t\geq 0}$ is a martingale defined  for $t\in \mathbb R_+ \cap[0, T^K]$ by  
 \begin{align*}
\mathcal M_1(t)= & \frac 1 N\int_0^t \int_{\mathbb R_+ \times \mathcal T}  \1_{\rho \le \nu_1(\mathbf U(s^-), \mathbf V(s^-), \theta)} \theta \, \widetilde{\N}_1( \dd s, \dd \rho, \dd \theta),
\end{align*}
where $\widetilde{\N}_1$ is the compensated Poisson measure of $\N_1$. This semi-martingale decomposition is the stochastic counterpart 
of the  semi-discrete problem \eqref{eq:discSKTsource} where the source ${\bf m}_1$ is random and given by the local martingale $\mathcal M_1$.

We can now prove the following moment estimates.
\begin{Prop} \label{prop:mom_p}
 Assume  that $b_i-d_i$ is bounded and $b_i$ is globally Lipschitz for $i\in \{1,2\}$. Let $p\geq 1$ and assume also  that \eqref{momunM} holds.
Then there exists $c \in \mathbb R$ such that for all $K \in \mathbb N$ and $M,N\geq 1$  and  $t\geq 0$, 
    $$ \E \left[ \norm{\mathbf U^{M,N} (t\wedge T^K)}_{1,M}^p \! + \norm{\mathbf V^{M,N} (t\wedge T^K)}_{1,M}^p \right] \lesssim e^{ct} \left(\! \E \left[ \norm{ \mathbf U^{M,N}(0)}_{1,M}^p\! + \! \norm{ \mathbf V^{M,N}(0)}_{1,M}^p \right]+ \frac {t} {N} \! \right) \! .$$
\end{Prop}
\begin{proof}
Fix $K \in \mathbb N$ and $p \ge 1$. We take the scalar product of SDE \eqref{eq:defPS} with the constant function $\mathbf 1$ and for any $t\leq T^K$,  
$$\norm {\mathbf U (t) }_{1,M} = \norm {\mathbf U_0}_{1,M} + \frac 1 N \int_0^{t} \int_{\mathbb R_+ \times \mathcal T}  \1_{\rho \le \nu_1(\mathbf U(s^-), \mathbf V(s^-), \theta)}\langle \mathbf 1 , \theta \rangle\, \N_1( \dd s, \dd \rho, \dd \theta).$$
For a jump of $\norm {\mathbf U (t) }_{1,M}$ given by $\frac 1 N\langle \mathbf 1 , \theta \rangle$, the corresponding jump of $\norm {\mathbf U (t) }_{1,M}^p$ is
$$( \norm {\mathbf U (t^-) }_{1,M} +\frac 1 N \langle \mathbf 1, \theta \rangle)^p -\norm {\mathbf U (t^-)}_{1,M}^p,$$
thus we have
\begin{align*}
    \norm {\mathbf U (t) }_{1,M}^p \! &= \norm {\mathbf U_0}_{1,M}^p \\
    & + \int_0^{t} \! \int_{\mathbb R_+ \times \mathcal T} \1_{\rho \le \nu_1(\mathbf U(s^-), \mathbf V(s^-), \theta)} [( \norm {\mathbf U(s^-) }_{1,M} +\frac 1 N \langle \mathbf 1, \theta \rangle)^p -\norm {\mathbf U(s^-)}_{1,M}^p] \, \N_1( \dd s, \dd \rho, \dd \theta).
    \end{align*}
Taking the expectation using that $T^K$ is a stopping time for the natural filtration, 
\begin{multline*}
    \E\left[\norm {\mathbf U (t\wedge T^K) }_{1,M}^p \right] \\
    = \E \left[\norm {\mathbf U_0}_{1,M}^p +   \int_0^{t\wedge T^K} \sum_{\theta \in  \mathcal T} \left(\left( \norm {\mathbf U (s)}_{1,M}  + \frac 1 N \langle \mathbf 1 ,\theta\rangle\right)^p - \norm {\mathbf U (s)}_{1,M}^p \right) \nu_1(\mathbf U(s), \mathbf V(s) , \theta)\, \dd s \right]. 
\end{multline*}
The mobility transitions have zero contribution, i.e. $\langle 1, \mathbf e_{i+\epsilon}-\mathbf e_{i}\rangle =0$. Then   for any $\mathbf u, \mathbf v \in \ell^2(\T_M)$,
\begin{multline*}
    \sum_{\theta \in  \mathcal T} \left(\left( \norm {\mathbf u}_{1,M}  + \frac 1 N \langle \mathbf 1 ,\theta\rangle\right)^p - \norm {\mathbf u}_{1,M}^p \right) \nu_1(\mathbf u, \mathbf v , \theta) \\
    =\sum_{j=1}^M  N u_j \Big[\left(\left( \norm {\mathbf u}_{1,M}  + \frac 1 {MN} \right)^p - \norm {\mathbf u}_{1,M}^p \right) b_1( u_j, v_j) \\
    +\left(\left( \norm {\mathbf u}_{1,M} - \frac 1 {MN} \right)^p - \norm {\mathbf u}_{1,M}^p \right) d_1( u_j, v_j) \Big]. 
\end{multline*}
We then use the inequalities $(x+ \varepsilon)^p - x^p \le x^p - (x- \varepsilon)^p + C_p \varepsilon^2 x^{p-2}$ and $x^p - (x- \varepsilon)^p \le C_p \varepsilon x^{p-1}$, which are given in  Proposition \ref{prop:ineg_p} of forthcoming Appendix, together with the upper bounds of  $b_1-d_1$  and  $b_1$ coming from assumptions. We  get \begin{align*}
&   \sum_{\theta \in  \mathcal T} \left(\left( \norm {\mathbf u}_{1,M}  + \frac 1 N \langle \mathbf 1 ,\theta\rangle\right)^p - \norm {\mathbf u}_{1,M}^p \right) \nu_1(\mathbf u, \mathbf v , \theta) \\
    &\leq \sum_{j=1}^M  N u_j \left[\left( \norm {\mathbf u}_{1,M}^p -\left( \norm {\mathbf u}_{1,M}  - \frac 1 {MN} \right)^p  \right)(b_1 - d_1)( u_j, v_j) + C_p \frac 1 {(MN)^2}  \norm {\mathbf u}_{1,M}^{p-2} b_1( u_j, v_j) \right] \\
    &\lesssim  \sum_{j=1}^M  N u_j \left[\frac {\rho_0} {MN}  \norm {\mathbf u}_{1,M}^{p-1} +  \frac 1 {(MN)^2}  \norm {\mathbf u}_{1,M}^{p-2}(1+ u_j + v_j) \right]\\
    &\lesssim \frac 1 M \sum_{j=1}^M  u_j \left[  \norm {\mathbf u}_{1,M}^{p-1}  +  \frac 1 {MN}  \norm {\mathbf u}_{1,M}^{p-2}(1+  M\norm {\mathbf u}_{1,M} +  M\norm {\mathbf v}_{1,M})\right]  \\
    &\le \norm {\mathbf u}_{1,M}\left[  \norm {\mathbf u}_{1,M}^{p-1}  +  \frac 1 {N}  \norm {\mathbf u}_{1,M}^{p-2}(1+  \norm {\mathbf u}_{1,M}+  \norm {\mathbf v}_{1,M}) \right].
\end{align*}
Using Young's inequality,
 $\norm {u}_{1,M}^{p-1}(1+  \norm {u}_{1,M}+  \norm {v}_{1,M}) \lesssim  1 +  \norm {u}_{1,M}^{p} + \norm {v}_{1,M}^{p}$, we obtain
\begin{align*}
& \E[\norm {\mathbf U (t\wedge T^K)}_{1,M}^p]\\
& \quad \le \E[\norm {\mathbf U (0)}_{1,M}^p] + C \int_0^t  \E[\norm {\mathbf U (s\wedge T^K)}_{1,M}^{p}] + \frac 1 N \left( 1+\E[\norm {\mathbf V (s\wedge T^K)}_{1,M}^{p}]\right)  \dd s,
\end{align*}
for some $c>0$.
Going through similar arguments for estimating $\mathbf V^K$, we sum the two estimates and get
$$\psi(t) \le \psi(0) + c \int_0^t \left(\psi(s) + \frac 1 {N} \right) \, \dd s ,$$
where
$$ \psi(t)=\E[\mathbf U (t\wedge T^K)^{p}] + \E[\mathbf V (t\wedge T^K)^{p}]. $$
This implies, by Gr\"onwall's lemma, that
$$\psi(t) \le e^{ct} \left(\psi(0)+ \frac {ct} {N} \right),$$
that is to say
$$ \E \left[ \norm{ \mathbf U(t\wedge T^K)}_{1,M}^p + \norm{ \mathbf V(t\wedge T^K)}_{1,M}^p \right] \le e^{ct} \left( \E \left[ \norm{ \mathbf U(0)}_{1,M}^p + \norm{ \mathbf V(0)}_{1,M}^p \right]+ \frac {ct} {N} \right).$$
\end{proof}

The previous proposition ensures that 
$$\lim_{K\rightarrow \infty }T^K=\infty \quad \text{a.s.},$$
since  for any $t\geq 0$ and a fixed initial condition with bounded first  moment,
$$\P(T^K\leq t) \leq \frac{ \E \left[ 1_{T^K\leq t} \left(\norm{\mathbf U^{M,N} (t\wedge T^K)}_{1,M} + \norm{\mathbf V^{M,N} (t\wedge T^K)}_{1,M}\right) \right]}{K}\lesssim \frac{Ce^{ct}}{K}$$
and the right hand side goes to $0$ as $K$ tends to infinity. It allows to construct the process $(\mathbf U^{M,N}, \mathbf V^{M,N})$ for any time $t\in \R_+$, and 
 $(\mathbf U^{M,N}, \mathbf V^{M,N})$ is the unique strong solution of $\eqref{eq:defPS}$ in $\mathbb D(\R^+,\mathbb N^{2M})$. Let us observe  that existence and pathwise uniqueness can then be extended  to any finite initial conditions by truncation procedure on the initial conditions, while we we are  using the moment estimates obtained  above for the proofs  of convergence.

\subsection{Martingales inequalities}
\label{Secmartingales}
We have shown in the previous part   that $T^K\rightarrow \infty$ a.s. as $K\rightarrow \infty$. Then   for $i\in\{1,2\}$, \begin{equation}
\label{expressM1}
    \mathcal M_i^{M,N}(t) = \frac 1 N \int_0^{t} \int_{\mathbb R_+ \times \mathcal T}  \1_{\rho \le \nu_i(\mathbf U^{M,N}(s^-), \mathbf V^{M,N}(s^-), \theta)} \theta \,  \widetilde \N_i (\dd s, \dd \rho, \dd \theta)
    \end{equation}
is   a local  martingale well defined for any $t\in \R_+$.
Reading the right-hand side of the estimate in Proposition \ref{prop:disc_stab}, we see that the approximation results will rely on bounds for the quantities ${ h}^M$. We denote ${ H}^{M,N}$ the stochastic counterpart, arising when the source term
 $\mathbf m^M$ is  the local martingale $\mathcal M^{M,N}$. Thus, $H_1^{M,N}$ is 
 given by  the following semi-martingale decomposition
\begin{align}\label{eq:defEstoc}
    H_1^{M,N}(t) &= 2 \mathcal Q_1^{M,N}(t)+\sum_{0 \le s \le t} \norm {\mathcal M_1^{M,N}(s) - \mathcal M_1^{M,N}(s^-)}_{-1,M}^2,
\end{align} 
where $\mathcal Q_1^{M,N}$ is the following local martingale
\begin{align}
\mathcal Q_1^{M,N}(t)&= -\int^t_0 \langle \mathbf Z^{M,N}(s) , \phi_1(\mathbf U^{M,N}(s), \mathbf V^{M,N}(s)) \rangle_{-1,M}\, \dd s \nonumber \\
& \hspace{5cm} +\sum_{0 \le s \le t} \langle \mathbf Z^{M,N}(s^-), \mathcal M_1^{M,N}(s) - \mathcal M_1^{M,N}(s^-) \rangle_{-1,M} \nonumber\\
    &= \frac 1 N \int_0^t\int_{\mathbb R_+ \times \mathcal T}  \1_{\rho \le \nu_1(\mathbf U^{M,N}(s^-), \mathbf V^{M,N}(s^-), \theta)} \langle \mathbf Z^{M,N}(s^-),\theta \rangle_{-1,M} \, \widetilde \N_1( \dd s, \dd \rho, \dd \theta)
    \label{Q1mPP}
    \end{align}
    and $\phi_1$ has been defined in \eqref{defphi1} and 
$$\mathbf Z^{M,N}:= \mathbf U^{M,N}-\widehat u^M.$$ 
The local martingale $ \mathcal Q_2^{M,N}$ and semimartingale $H_2^{M,N}$ for the second species are defined similarly.
We  prove now that these local martingales  are square integrable martingales whose moments  are bounded by powers of $1/N$ in $H^{-1}$.
\begin{Prop} \label{prop:plusDeMart} 
For  $i\in \{1,2\}$,  assume  that $b_i-d_i$ is bounded and $b_i$ and $d_i$ are globally Lipschitz.\\
i)  Assume that   initial conditions satisfy \eqref{momunM} with $p=2$. Then, for each $i\in \{1,2\}$,  $\mathcal M_i^{M,N}$ is a square-integrable martingale in $(\R_+)^{2M}$ and  for any $T\geq 0$ and $M,N\geq 1$,
    $$\E \left[\sup_{t \le T} \norm{\mathcal M_i^{M,N}(t)}_{-1,M}^2 \right] \lesssim \frac 1 N \E \left[\int_\QTM 1 + (\mathbf U^{M,N})^2 +(\mathbf V^{M,N})^2\right].$$
ii) Let $p\geq 2$ and assume that initial conditions satisfy \eqref{momunM}. Then, for each  $i\in \{1,2\}$, 
$\mathcal Q_i^{M,N}$ 
    is a square-integrable martingale in $\R$ and for any $T\geq 0$, $M,N\geq 1$,
  for any $T\geq 0$,
$$ \E \left[\sup_{t \le T} \vert \mathcal Q_i^{M,N} \vert (t)^{p/2} \right]  \lesssim    \left(\frac {1+T} N \right)^{\frac p 4}\E[1 + |||\mathbf U^{M,N}|||_{T,M}^{p} + |||\mathbf V^{M,N}|||_{T,M}^{p}] .$$
\end{Prop}
\begin{proof} Let us first control the jumps  of the local martingales in $H^{-1}$. Jumps can be birth, deaths (BD) and motion (M) and we recall that they are given by $\mathcal T$ : 
$$\mathcal T = \mathcal T_{BD} \cup \mathcal T_{M}, \quad  \mathcal T_{BD} = \bigcup_{j=1}^M \left\{\frac 1 N \mathbf e_j, -\frac 1 N \mathbf e_j\right\}, \quad \mathcal T_{M} = \bigcup_{j=1}^M \left\{\frac 1 N (\mathbf e_{j+1} - \mathbf e_j), \frac 1 N (\mathbf e_{j-1} - \mathbf e_j)\right\}.$$
We first deal with births and deaths and use the first part of Proposition \ref{prop:jumpSize} to for  $\norm{ \theta}_{-1,M}$ and get
\begin{align*}
    &\int_{\mathbb R_+ \times \mathcal T_{BD}} \1_{\rho \le \nu_1(\mathbf u, \mathbf v, \theta)} \norm{\frac 1 N  \theta}^2_{-1,M} \,  \dd \rho \dd \theta \\
    &\qquad =\frac 1 {N^2M}\int_{\mathbb R_+ \times \mathcal T_{BC}} \1_{\rho \le \nu_1(\mathbf u, \mathbf v, \theta)} \,  \dd \rho \dd \theta  = \frac 1 {N^2M}\sum_{\theta \in \mathcal T_{BC}}\nu_1(\mathbf u, \mathbf v, \theta) \\
    &\qquad = \frac 1 {N^2M}\sum_{j=1}^M Nu_j b_1(u_j, v_j) + Nu_j d_1(u_j, v_j)
    \lesssim\frac 1 {NM}\sum_{j=1}^M 1 + u_j^2 + v_j^2 = \frac 1 {N}(1 + \norm{\mathbf u}_{2,M}^2 +\norm{\mathbf v}_{2,M}^2),
\end{align*}
the last inequality coming from the fact that $b_1$ and $d_1$ are Lipschitz and Young's inequality.
Similarly, the second part of   Proposition \ref{prop:jumpSize} yields a control of the motion part :
\begin{align*}
    &\int_{\mathbb R_+ \times \mathcal T_{\mu}} \1_{\rho \le \nu_1(\mathbf u, \mathbf v, \theta)} \norm{\frac 1 N  \theta}^2_{-1,M} \,  \dd \rho \dd \theta \\
    &\qquad =\frac 1 {N^2M^3}\int_{\mathbb R_+ \times \mathcal T_{\mu}} \1_{\rho \le \nu_1(\mathbf u, \mathbf v, \theta)} \,  \dd \rho \dd \theta  = \frac 1 {N^2M^3}\sum_{\theta \in \mathcal T_{\mu}}\nu_1(\mathbf u, \mathbf v, \theta)
    = \frac 1 {N^2M^3}\sum_{j=1}^M NM^2u_j \mu_1(v_j) \\
    &\qquad \lesssim \frac 1 {NM}\sum_{j=1}^M u_j(1 + v_j)
    \lesssim\frac 1 {NM}\sum_{j=1}^M 1 + u_j^2 + v_j^2 = \frac 1 {N}(1 + \norm{\mathbf u}_{2,M}^2 +\norm{\mathbf v}_{2,M}^2).
\end{align*}
Proceeding similarly for the second species and 
assembling these two  bounds for the jumps yields
\begin{equation}\label{eq:Jbound2}
 \int_{\mathbb R_+ \times \mathcal T} \1_{\rho \le \nu_i(\mathbf u, \mathbf v, \theta)} \norm{\frac 1 N  \theta}^2_{-1,M} \,  \dd \rho \dd \theta \lesssim \frac 1 {N}(1 + \norm{\mathbf u}_{2,M}^2 +\norm{\mathbf v}_{2,M}^2) 
\end{equation}
for any $\mathbf u,\mathbf v \in \ell^2(\T_M)$ and $i\in \{1,2\}$.
Replacing $(\mathbf u,\mathbf v)$ by $(\mathbf U(s), \mathbf V(s))$ in this expression and   integrating in time and taking the expectation, we obtain  for any  $M,N\geq 1$ and $t\in [0,T]$,
\begin{equation} \label{eq:1534}
   \E \left[\int_0^t \int_{\mathbb R_+ \times \mathcal T} \norm{\frac 1 N \1_{\rho \le \nu_i(\mathbf U(s), \mathbf V(s), \theta)} \theta}^2_{-1,M} \,  \dd \rho \dd \theta \right] \lesssim \frac 1 N \E \left[\int_\QTM 1 + \mathbf U^2 +\mathbf V^2\right].
\end{equation}
Moreover,
$$\int_\QTM \mathbf U^2 +\mathbf V^2 = \int_0^T \norm{\mathbf U(s)}_{2,M}^2 + \norm{\mathbf V(s)}_{2,M}^2  \, \dd s \lesssim M\int_0^T \norm{\mathbf U(s)}_{1,M}^2 + \norm{\mathbf V(s)}_{1,M}^2  \, \dd s,$$
and we know from we know from Proposition  \ref{prop:mom_p} that the second moment on the left hand side is finite. It ensures that for each $i\in \{1,2\}$,
$$\norm{\mathcal M_i(t)}_{-1,M}^2-\int_0^{t} \int_{\mathbb R_+ \times \mathcal T}
    \norm{\frac 1 N \1_{\rho \le \nu_i(\mathbf U(s), \mathbf V(s), \theta)} \theta}^2_{-1,M} \,\dd s \dd \rho \dd \theta$$
    is a real martingale  for any $t\geq 0$ and $\mathcal M_i$ is a square-integrable martingale in $(\R_+)^{2M}$.
We  work here with $\norm{\cdot}_{-1,M}$ rather than $\norm{\cdot}_{2,M}$, which doesn't change   the integrability properties of $\mathcal M_i$ since $\R_+^{2M}$ is   finite-dimensional and refer to Section II-3 of \cite{ikeda}.
Using \eqref{eq:1534}, we can bound the second moment of $\mathcal M_i(t)$:
$$\E \left[\norm{\mathcal M_i(t)}_{-1,M}^2 \right] \lesssim \frac 1 N \E \left[\int_\QTM 1 + \mathbf U^2 +\mathbf V^2\right].$$
Since $\mathcal M_i$ is a martingale, $\norm{\mathcal M_i(t)}_{-1,M}$ is a submartingale and Doob's inequality  proves  part $i)$ of the lemma.\\

Let us now consider the local martingale $\mathcal Q_1$.
The domination of $\mathbf Z$ by ${\bf U}$ and expression \eqref{Q1mPP} 
 ensure that $\mathcal Q_1$ is a square-integrable martingale whose quadratic variation is
\begin{align*}
    \langle \mathcal Q_1 \rangle (t) 
     &=  \int_0^t\int_{\mathbb R_+ \times \mathcal T}  \1_{\rho \le \nu_1(\mathbf U(s), \mathbf V(s), \theta)}\left\langle \mathbf Z(s), \frac 1 N \theta   \right\rangle_{-1,M}^2 \,\dd s \dd \rho \dd \theta. 
     \end{align*}
  More precisely, using successively Cauchy-Schwarz inequality and  \eqref{eq:Jbound2}
   \begin{align*}
   \langle \mathcal Q_1 \rangle (t)  &\le \sup_{0 \le s\le t}\norm{\mathbf Z(s)}_{-1,M}^2  \int_0^t \int_{\mathbb R_+ \times \mathcal T}  \1_{\rho \le \nu_1(\mathbf U(s), \mathbf V(s), \theta)} \norm{ \frac 1 N \theta}_{-1,M}^2 \, \dd \rho \dd \theta  \,\dd s \\
    &\lesssim \sup_{0 \le s\le t}(1+\norm{\mathbf U(s)}_{-1,M}^2 )\int_0^t\frac {1} N (1 +\norm {\mathbf U (s)}^2_{2,M} +\norm {\mathbf V (s)}^2_{2,M}) \,\dd s, \\
    &\lesssim \frac 1 N (1 + |||\mathbf U|||_{t,M}^2)(t+|||\mathbf U|||_{t,M}^2 + |||\mathbf V|||_{t,M}^2) \lesssim (1+ t)(1 + |||\mathbf U|||_{t,M}^4 + |||\mathbf V|||_{t,M}^4).
\end{align*}
   Raising to the power $p/4$ and taking the expectation give
$$\E[\langle \mathcal Q_1 \rangle (T)^{\frac p 4}] \lesssim \frac 1 {N^{\frac p 4}}(1+ T)^{\frac p 4}\E[1 + |||\mathbf U|||_{T,M}^{p} + |||\mathbf V|||_{T,M}^{p}]. $$
Since $\mathcal Q_1$ is a  martingale starting from $0$,   Burkholder-Davis-Gundy inequality yields the second part of the proposition.
\end{proof}

We can now obtain the expected estimations on $H_i^{M,N}$, allowing for  convergence of martingale pertubations to $0$ as $N$ tends to infinity, uniformly  with respect to $M$.
\begin{Lemme} \label{lem:boundE1}   For  $i\in \{1,2\}$,  assume  that $b_i-d_i$ is bounded and $b_i$ and $d_i$ are globally Lipschitz.\\
Let $p\ge 2$ and assume that the initial conditions satisfy \eqref{momunM}.
Then  for any $T\geq 0$ and  $N,M\geq 1$,
$$ \E \left[\sup_{t \le T} |H_i^{M,N}(t)|^{p/2} \right] \lesssim \left(\frac {1+T^2} N \right)^{\frac p 4}\E[1 + |||\mathbf U^{M,N}|||_{T,M}^{p} + |||\mathbf V^{M,N}|||_{T,M}^{p}]. $$
\end{Lemme}
\begin{proof} We write $q=p/2$.
Proposition \ref{prop:plusDeMart} allows to control $\mathcal Q_1^q$ and we need now to estimate the second term of $H_1^q$, i.e. $\mathcal P_1^q$, where  
 \begin{align*}
\mathcal P_1(t):= \sum_{0 \le s \le t} \norm {\mathcal M_1(s) - \mathcal M_1(s^-)}_{-1,M}^2 &= \sum_{0 \le s \le t} \norm {\mathbf U(s) - \mathbf U(s^-)}_{-1,M}^2\\
&= \frac 1 {N^2} \int_0^{t} \int_{\mathbb R_+ \times \mathcal T}  \1_{\rho \le \nu_1(\mathbf U(s^-), \mathbf V(s^-), \theta)}\norm{ \theta}^2_{-1,M} \,  \N_1(\dd s ,\dd \rho ,\dd \theta ).
\end{align*}
First, 
  $$\mathcal P_1(t)^q = \int_0^{t} \int_{\mathbb R_+ \times \mathcal T}  \1_{\rho \le \nu_1(\mathbf U(s^-), \mathbf V(s^-), \theta)}\left[\left(Q(s^-) + \frac 1 {N^2} \norm{ \theta}^2_{-1,M}\right)^q-Q(s^-)^q \right]\,  \N_1(\dd s ,\dd \rho ,\dd \theta ) .$$
    Using the inequality $(x+a)^q-x^q\lesssim ax^{q-1}+ a^q$ from Proposition \ref{prop:ineg_p} in the appendix we deduce
$$\mathcal P_1^q (t)\le  \int_0^{t} \int_{\mathbb R_+ \times \mathcal T}  \1_{\rho \le \nu_1(\mathbf U(s^-), \mathbf V(s^-), \theta)}\left[\frac 1 {N^2} \norm{ \theta}^2_{-1,M} \mathcal P_1(s^-)^{q-1} +  \frac 1 {N^{2q}} \norm{ \theta}^{2q}_{-1,M}\right]\,  \N_1(\dd s ,\dd \rho ,\dd \theta ).  $$
Bounding $\norm{ \theta}^{2(q-1)}_{-1,M}$ from the last term by $ 1 /{M^{q-1}}$ thanks to Proposition \ref{prop:jumpSize},
\begin{multline*}
    \mathcal P_1(t)^q \lesssim  \frac 1 {N^2} \int_0^{t} \int_{\mathbb R_+ \times \mathcal T}  \1_{\rho \le \nu_1(\mathbf U(s^-), \mathbf V(s^-), \theta)}\norm{ \theta}^2_{-1,M} \mathcal P_1(s^-)^{q-1}  \,  \N_1(\dd s ,\dd \rho ,\dd \theta ) \\
    + \frac 1 {N^{2q} M^{q-1}}\int_0^{t} \int_{\mathbb R_+ \times \mathcal T}  \1_{\rho \le \nu_1(\mathbf U(s^-), \mathbf V(s^-), \theta)}  \norm{ \theta}^{2}_{-1,M}\,  \N_1(\dd s ,\dd \rho ,\dd \theta ). 
\end{multline*}
Taking the expectation and using \eqref{eq:Jbound2} and $M\geq 1$,
\begin{align*}
    \E[\mathcal P_1(t)^q] & \lesssim  \frac 1 {N}  \E \left[\int_0^t (1+ \norm{\mathbf U(s)}_{2,M}^2 +\norm{\mathbf V(s)}_{2,M}^2 ) \mathcal P_1(s)^{q-1}\,  \dd s  \right]\\
  &\qquad   + \frac 1 {N^{2q-1}}  \E \left[\int_0^t (1+ \norm{\mathbf U(s)}_{2,M}^2 +\norm{\mathbf V(s)}_{2,M}^2 )\,  \dd s \right].
    \end{align*}
Since $\mathcal P_1$ is non-decreasing, 
$$\E[\mathcal P_1(t)^q] \lesssim  \frac 1 {N}  \E \left[ \mathcal P_1(t)^{q-1} \int_\QtM 1+\mathbf U^2 +\mathbf V^2 \right]+ \frac 1 {N^{2q-1} }  \E \left[\int_\QtM 1+\mathbf U^2 +\mathbf V^2  \right].$$
This ends the proof for $q=1$ (i.e. $p=2$). For $q>1$, 
using Young's inequality $xy \lesssim (\varepsilon x)^{\frac q {q-1}} + (y/ \varepsilon)^q$,
$$\E[\mathcal P_1(t)^q] \lesssim   \varepsilon^{\frac q {q-1}} \E [ \mathcal P_1(t)^{q}] 
+\left(\frac {\varepsilon^{-q}} {N^{q}}  +\frac 1 {N^{q} } \right)   \E \left[\left(\int_\QtM 1+\mathbf U^2 +\mathbf V^2 \right)^q  \right] + \frac 1 {N^{(2q-2)\frac q {q-1}}}.$$
Fixing $\varepsilon >0$ small enough to absorb $\E [ \mathcal P_1(t)^{p}]$ from the right-hand side in the left-hand side, we get
$$\E[\mathcal P_1(t)^q] \lesssim  \frac {1} {N^{q}} \,    \E \left[\left(\int_\QtM 1+\mathbf U^2 +\mathbf V^2 \right)^q \right] + \frac 1 {N^{2q}}$$
Since $N \ge 1$ and 
$$ \int_\QtM 1+\mathbf U^2 +\mathbf V^2 \le t + |||\mathbf U|||_{t,M}^2 + |||\mathbf V|||_{t,M}^2,$$ 
this concludes the proof.
\end{proof}

\subsection{Large deviations estimates for births and deaths}
\label{sec:nJumps}

In this section, we  control the number of births and deaths along time together with the corresponding rates cumulated over time.
The first step is to compare the number of jumps of interest to their cumulative intensity over time, when time is becoming large. We provide    large deviations estimates for the gap, for general jump Markov processes, which can be of independent interest. We then apply it for our framework. \\

We consider a general  jump process Markov $X$ on a discrete countable space $E$ and write $\nu(x,y)$ the jump rate from $x$ to $y$, with  $\nu(x,x)=0$ by convention. Our aim is to control terms of the form
$$I_A(t)=\int_0^t  \nu_A(X_s)  \, \dd s,$$
where $A$ is a subset of  the set of transitions $\{(x,y) :  x\in E, y \in  E, x\ne y\}$ and 
$$\nu_A(x)= \sum_{(x ,y)\in A}  \nu(x,y)$$
is the total rate of transitions of type $A$ when the process is in state $x$.
The process $I_A$ term quantifies the cumulative transition rates of type $A$ along trajectories of $X$. It appears in the semi martingale decomposition of jump Markov processes. The purpose of this subsection is to prove that this term is close to the cumulative number of transitions of type $A$, defined by :
$$N_A(t)=\sum_{s\leq t} 1_{(X_{s-},  X_s)\in A}.$$
More precisely, we prove   that $I_A$ and $N_A$ have to be comparable with a great probability as soon as one of them is large, in the following sense. We consider the times when $N_A$ or $I_A$ is larger than a given value $K$ :
$$\mathfrak I^K_A=\{t \in \R_+ \, : \,  \max(N_A(t),I_A(t))\geq K\}$$
and prove the following  large deviation estimate. This estimate does not need assumption on the ergodicity and asymptotic behavior of the process.
\begin{Prop} \label{lem:nJumps}
    There exists two universal constants $c,C$ such that for all $\varepsilon \in (0,\frac 1 2)$ and $K >0$,
    $$ \P\left( \exists t \in \mathfrak I^K_A : |N_A(t) -I_A(t)| \ge \varepsilon I_A(t)  \right) \le  C \left(1+\frac 1 { \varepsilon^3 K}\right) e^{-c \varepsilon^2 K}.$$
\end{Prop}
We first prove this result for Poisson process and then extend it by using the suitable time change. Indeed,
the identity
$$Y_{I_A(t)}=X(t)$$
defines a jump Markov process $Y$ as soon as $X$ does not reach a point where $\nu_A$ vanishes. Moreover the process counting the jumps of $Y$ in $A$ is a Poisson process with parameter $1$. Authorizing $\nu_A$ to vanish requires an additional approximation argument and a more involved time change, see the proof below.
\begin{proof}
We first consider the case when $X=P$ is a Poisson process with parameter $1$. The jumps of $P$ are of size $1$ and counting  all its jumps  $J=\{(n,n+1) : n\geq 0\}$ until time $t$  yields the value of the process at time $t$ : $P(t)=N_J(t)$. Moreover $\nu_J=1$ and $I_J(t)=t$ and $ \mathfrak I^K_J=\{ t\in \R_+ : P(t)\geq K\}\cup [K,\infty)$.
We first estimate
\begin{equation}
\label{firstesgd}
\P(\mathfrak I^K_J \cap [0,(1-\varepsilon)K] \ne \varnothing)\leq \P(P((1-\varepsilon)K)\ge K) \le e^{-c\varepsilon^2K}
\end{equation}
using a classical large deviation estimate for Poisson random variable, see for instance Theorem 1.1 in   \cite{ultralogconcave}.
Second, we split the remaining time interval
$$[(1-\varepsilon) K , \infty ) = \bigcup_{j =0}^\infty [t_j, t_{j+1}], \quad \text{where }  \, \, t_j:= \beta^j (1-\varepsilon) K$$
and
$$ \beta=\frac{1}{2}\left\{1+\min\left( \frac{1 + \varepsilon}{1+  \varepsilon /2},  \frac{1 - \varepsilon/2}{1-  \varepsilon}\right)\right\} \, \in \, \left(1  ,\min\left( \frac{1 + \varepsilon}{1+  \varepsilon /2},  \frac{1 - \varepsilon/2}{1-  \varepsilon}\right)\right).$$
This choice of $\beta$ and monotonicity of $P$ implies for any $t\in [t_j,t_{j+1}]$,
$$\{P(t) \le (1 - \varepsilon)t\} \subset \{P(t_{j}) \le (1 - \varepsilon) t_{j+1}\}\subset \{P(t_j)-t_j \le \varepsilon t_j/2\}$$
and
$$\{ P(t) \ge (1 + \varepsilon)t\}\subset\{P(t_{j+1}) \ge (1 + \varepsilon) t_j \} \subset \{P(t_{j+1})-t_{j+1} \ge \varepsilon t_j/2\}.$$ 
Combining these inclusions of events
ensures that
$$ \{\exists t \in [ (1-\varepsilon)K, \infty) : |P(t) - t| \ge \varepsilon t \} \subset 
 \bigcup_{j=0}^\infty \{|P(t_j) - t_j| \ge  \varepsilon  t_j /2 \}.$$
We can use again large deviation (tail) estimate for Poisson law  and get
\begin{align*}& \P(\exists t \in [ (1-\varepsilon)K, \infty) : |P(t) - t| \ge \varepsilon t )\\
& \qquad  \leq \sum_{j= 0}^{\infty} \P(|P(t_j) - t_j| \ge  \varepsilon  t_j /2) \leq 
\sum_{j=0}^\infty e^{-c \varepsilon^2 t_j} 
\lesssim \frac 1 { \varepsilon^2 t_0 \ln \beta} e^{-c \varepsilon^2 t_0},
\end{align*}
where $\lesssim$ means here that  the constant is universal (it depends neither on $K$ nor on $\varepsilon$) and
this  inequality comes from the following computation for  $\alpha >0$ and $\beta >1$, using the change of variables $y = \alpha \beta^x$,
$$ \sum_{j=0}^\infty e^{- \alpha \beta^j} \le \int_{-1}^\infty e^{- \alpha \beta^x} \, \dd x = \frac 1 {\ln \beta}\int_{\alpha/ \beta}^\infty \frac{e^{- y}}{y} \, \dd y \le \frac \beta { \alpha \ln \beta} \int_{\alpha/ \beta}^\infty e^{- y} \, \dd y =\frac \beta { \alpha \ln \beta} e^{- \alpha/ \beta}.$$ 
Combining this estimate with \eqref{firstesgd} and using  $t_0 \ln \beta \gtrsim  \varepsilon K$ yields
 the expected result for Poisson process :
  \begin{align} & \P( \exists t \in \mathfrak I^K_J : |P(t) - t| \ge \varepsilon t  ) \nonumber \\
  &\qquad\leq \P(  \mathfrak I^K_J \cap[0,(1-K)\varepsilon] \ne \varnothing)  + \P( \exists t \in [ (1-\varepsilon)K, \infty) : |P(t) - t| \ge \varepsilon t  )\nonumber\\
  & \qquad \lesssim   \left(1+\frac 1 { \varepsilon^3 K}\right) e^{-c \varepsilon^2 K}. \label{Poissonborn}
  \end{align}
$\newline$
We now prove the result for any jump process $X$ and set of transitions $A$ such that $\inf_{x\in E} \nu_A(x)>0$. In that case, $I_A(t)=\int_0^t  \nu_A(X_s)  \, \dd s$ increases continuously and at least linearly, thus realizing a bijection from $[0,\infty)$ into $[0,\infty)$ and we define its inverse function
$$\tau_A(t)=I_A^{-1}(t)=\inf\{ u \in \R_+ :  \int_0^u  \nu_A(X_s)  \, \dd s\geq t\}.$$
The proof now relies on the fact that the process  $P$ defined for $t\geq 0$ by
$$P(t)=N_A(\tau_A(t))$$
is a Poisson process with parameter $1$, see for instance Chapter 6 in \cite{EK}.
Then $P(u)=N_A(t)$ with $u=I_A(t)$ and
$$\{\exists u \in [0,K] \cup  P^{-1}([K,\infty)) \,    : \,  |P(u) - u| \ge \varepsilon u \}= \{ \exists t \in \mathfrak I^K_A : |N_A(t) - I_A(t)| \ge \varepsilon I_A(t)  \}.$$
Then we can use  $\eqref{Poissonborn}$ for the left hand side and control the right hand side and  get
 \begin{align}  \P( \exists t \in \mathfrak I^K_A : |N_A(t) - I_A(t)| \ge \varepsilon I_A(t)  ) &\lesssim   \left(1+\frac 1 { \varepsilon^3 K}\right) e^{-c \varepsilon^2 K}. \label{numinborn}
  \end{align}

To conclude the proof, we need to deal with the case when $\inf_{x\in E} \nu_A(x)=0$, which will be useful latter.
For that purpose, we introduce a new Poisson process $P^{\alpha}$ with parameter $\alpha>0$, independent of $X$. We consider  the jump Markov process $(X,P_{\alpha})$ in $E\times \mathbb N$ and consider the following set of transitions 
$$B=\{((x,n),(y,n)) \, : \, (x,y)\in A, n\geq 0\}\cup \{((x,n),(x,n+1)) \, : \, x \in E, n\geq 0\},$$
which means that either the first coordinate $X$  makes a jump in $A$ or the second coordinate $P_{\alpha}$ jumps. We observe that $\nu_B(x,n)=\nu_A(x)+\alpha$ for any $(x,n)\in E\times \mathbb N$. Thus, for this  jump Markov process $(X,P_{\alpha})$, $\inf_{(x,n)\in E\times \mathbb N} \nu_B(x,n)\geq  \alpha>0$
and we can invoke \eqref{numinborn}. We get
for any $T\geq 0$,
 \begin{align}
 \label{majrecup}
 \P( \exists t \in \mathfrak I^K_B \cap[0,T] : |N_B(t) - I_B(t)| \ge \varepsilon I_B(t) ) &\lesssim   \left(1+\frac 1 { \varepsilon^3 K}\right) e^{-c \varepsilon^2 K}. 
  \end{align}
Observe that the right hand side does not depend on $\alpha$ or $T>0$.  We also recall that $N_A(t)=\sum_{s\leq t} 1_{(X_{s-},  X_s)\in A}$ and here $N_B(t)=\sum_{s\leq t} 1_{(X_{s-},  X_s)\in A}+1_{P^{\alpha}(s-)\ne P^{\alpha}(s)}$,
$$N_B(t)=N_A(t)+P^{\alpha}(t), \quad I_B(t)=I_A(t)+\alpha t$$
and
$$ \mathfrak I^K_B\cap[0,T]=\{ t\in [0,T] : \max(N_A(t)+P^{\alpha}(t),I_A(t)+\alpha t) \geq K\}.$$
Adding that $P^{\alpha}$ converges to $0$ in $\mathbb D([0,T], \mathbb N)$ and similarly $(\alpha t)_{t\in [0,T]}$ converges uniformly to $0$ in $[0,T]$ as $\alpha$ tends to $0$, we can let $\alpha$ go to zero in \eqref{majrecup} and get
\begin{align*}  \P( \exists t \in \mathfrak I^K_A \cap[0,T] : |N_A(t) - I_A(t)| \ge \varepsilon I_A(t)  ) &\lesssim   \left(1+\frac 1 { \varepsilon^3 K}\right) e^{-c \varepsilon^2 K}.  \end{align*}
Letting $T$ go to infinity ends the proof. \end{proof}
A direct consequence of this estimate is that the distribution of $N_A$ can be controlled  by $I_A$, which induces a bound on the moments of $N_A$ expressed with analogue moments of $I_A$

We are now using the previous estimates for the jump Markov process $X= (\mathbf U^{M,N}, \mathbf V^{M,N})$ describing the evolution of species.  We recall from the first part of Hypothesis \ref{hyp:bd} $ii)$, that there  exists $a,c>0$ and $\alpha \in (0,1)$  such that for any $u,v\geq 0$,
\begin{equation}
\label{condtauxnm}
    b_1(u,v)\leq \rho_0+
 \alpha \, d_1(u,v) .
\end{equation}
It will allow us to control births terms by deaths terms.
Keeping the notation above, we consider the renormalized cumulated intensity of births of the first species
$$ I_B^{M,N}(t) =  \frac{1}{M}\int_{0}^t \left(\sum_{i=1}^{M}  U_i^{M,N}(s) \, b_1(U_i^{M,N}(s),  V_i^{M,N}(s))\right) ds.$$
 Similarly the deaths are evaluated by  
$$ I_D^{M,N}(t) = \frac{1}{M} \int_{0}^t \left(\sum_{i=1}^{M}  U_i^{M,N}(s)\,  d_1(U_i^{M,N}(s),  V_i^{M,N}(s))\right) ds.$$
 We observe that the death term $I_D$ dominates the birth term $I_B$, due to the domination of $b$ by $d$ coming from \eqref{condtauxnm}, so the delicate part is the control of $I_D$. Using the previous result, $I_D$ can be related to the number of deaths, while there cannot be too many deaths. Indeed  this would require a comparable number of births in the first place but for large populations the hypothesis on birth and death rates prevents the births to compensate the deaths. Combining these arguments, we can prove the following large deviation estimate. We only focus here on the first species. An analogous result holds  for the second species and more general structured population models.  \begin{Prop}
\label{gdevIMU}Assume \eqref{condtauxnm}.
Then, there exists constants $t_1>0$ and $\kappa > 1 $ and $\gamma \in (0,1)$  such that for all $M,N\geq 1$ and for any $K_0 \geq 1$ and for  any initial condition  satisfying  $\mathbf [{\bf U}^{M,N}_0]_M \le K_0 $ a.s., we have
$$ \P \left( I_B^{M,N}(t_1)+ I_D^{M,N}(t_1)  \ge \kappa K_0 \right) +\P \left( \sup_{s \le t_1} [\mathbf U^{M,N}(s)]_M  \ge \kappa K_0 \right) \lesssim \gamma^{K_0MN}.$$
Moreover, the constant behind $\lesssim$  depends only on $\rho_0, \alpha$ and not on $K_0, M$ or $N$.
\end{Prop}
We observe in the proof that $\kappa$ can be chosen large enough, and the constant behind $\lesssim$ decreases with respect to $\kappa$. 
\begin{proof}
Let us  consider  $N_B$ (resp. $N_D$) the process counting the number of births (resp. deaths) along time for the first species $\mathbf U$.  To be alive or die at some time, an individual has to be born before or to already exist at the initial time,  so  for any $t\geq 0$,
\begin{equation} \label{eq:boundND}
    \max( N_D(t) ,  MN  [\mathbf U(t)]_M) \leq  N_B(t) + MN [\mathbf U_0]_M \leq  N_B(t) + MNK_0 \qquad \text{a.s.}
\end{equation} 
for any initial condition satisfying $[\mathbf U_0]_M \le K_0$.
Moreover, applying   \eqref{condtauxnm} to the process of population densities  and integrating it  along time yields
\begin{equation}
\label{txBD}
     I_B(t) \le \alpha I_D(t) + \rho_0  \int_0^t [\mathbf U(s)]_M \, \dd s \le \alpha I_D(t) + \rho_0  t\left(\frac{N_B(t)}{MN}+K_0\right) \qquad \text{a.s.}
     \end{equation}
     since $N_B$ is non decreasing.
We need to control the tail of $N_B$ and the other bounds will follow using the two previous inequalities, together with the link between $I_D$ and $N_D$ obtained above via large deviations estimates. 
To control $N_B$, we split its tail as follows 
\begin{align}
\label{ecureuil}
     \left\{ N_B(t) \ge \kappa K_0 MN \right\} &
          \subset A_{t,\varepsilon,\kappa} 
     \cup B_{t,\varepsilon,\varrho, \kappa} \cup C_{t,\varrho, \varepsilon} \cup D_{t,\varepsilon,\kappa}
     \end{align}
     where $t, \varrho, \kappa$ are positive real numbers and $\varepsilon \in (0,1)$ that will be fixed later and 
     \begin{align*}
    A_{t,\varepsilon, \kappa}&=\{  N_B(t) \ge \kappa K_0 MN \, ; \, N_B(t) \geq (1 + \varepsilon) MNI_B(t)\}\\
B_{t,\varepsilon,\varrho, \kappa}&=   \left\{N_B(t) \ge \kappa K_0 MN   \, ; \,  N_B(t) \leq (1 + \varepsilon) MNI_B(t)\, ; \,  I_D(t) <  \varrho K_0
 \right\} \\
 C_{t,\varepsilon,\kappa}&=\left\{N_B(t) \ge \kappa K_0 MN \, ; \, N_B(t) \leq (1 + \varepsilon) MNI_B(t) \, ;  \,  MNI_D(t) \leq (1+\varepsilon)N_D(t)
 \right\}\\
D_{t,\varrho, \varepsilon}& =   \left\{ I_D(t) \geq \varrho K_0 \, ; \, MNI_D(t) \geq (1+\varepsilon) N_D(t)\right\}.
\end{align*}     
Let us deal with these four terms. Roughly, the estimates for $A$ and $D$ will come from large deviations estimates of  Proposition \ref{lem:nJumps}, while the two other will be due to \eqref{eq:boundND} and \eqref{txBD}.
First,  using  Proposition \ref{lem:nJumps} for births of $\bf U$ with $K=\kappa K_0MN$ ensures that for any $t\geq 0$, $\kappa>1$ and $N,M \geq 1$ and $K_0>0$,
$$\P\left(A_{t,\varepsilon,\kappa}  \right) \leq c_{\varepsilon, \kappa, K_0 } \gamma_{\varepsilon, \kappa}^{K_0MN}, $$
where $\gamma_{\varepsilon, \kappa} \in (0,1)$ and $c_{\varepsilon, \kappa, K_0 }= C \left(1+ 1/({ \varepsilon^3 \kappa K_0})\right).$ Furthermore, using   \eqref{txBD}, we observe that
\begin{align*}
&\left\{ N_B(t) \leq (1 + \varepsilon) MNI_B(t) \right\}\\ &\qquad \subset
\left\{ N_B(t)(1-(1 + \varepsilon)t\rho_0) \leq (1 + \varepsilon)( \alpha MNI_D(t) + t\rho_0  MN K_0) \right\}.
\end{align*}
Then 
$$B_{t,\varepsilon,\varrho,\kappa}\subset \{\kappa(1-t\rho_0- \varepsilon t\rho_0)\leq (1 + \varepsilon) (\alpha\varrho +t \rho_0 ) \},$$
which ensures that
$$B_{t_1,\varepsilon,\varrho,\kappa}=\varnothing,  \quad \text{for  } \, t_1=\frac{1-\sqrt{\alpha}}{\rho_0}, \, 
\varrho=\frac{\kappa\frac{\sqrt{\alpha}-\varepsilon(1-\sqrt{\alpha})}{1+\varepsilon}-1+\sqrt{\alpha}}{2\alpha}$$
soon as $\varrho>0$, i.e. for  any   $\kappa  \in (0,\infty)$ large enough and $\varepsilon \in  (0,1)$ small enough.
Furthermore,  combining 
\eqref{eq:boundND} and \eqref{txBD} yields
$$C_{t_1,\varepsilon,\kappa}\subset \left\{ N_B(t_1) \ge \kappa K_0 MN \, ; \, \left(\frac{1}{1+\varepsilon}-\rho_0 t_1-\alpha(1+\varepsilon)\right)N_B(t_1)\leq (\varrho_0t_1+\alpha(1+\varepsilon))MNK_0 \right\}.
$$
Since $\alpha<1$,  $1-\rho_0 t_1=\sqrt{\alpha} >\alpha$ and we can also choose $\varepsilon \in (0,1)$ (small enough) and then $\kappa>0$ (large enough) such that
$$  \frac{1}{1+\varepsilon}-\rho_0 t_1-\alpha(1+\varepsilon)>\frac{\rho_0t_1+\alpha(1+\varepsilon)}{\kappa},$$
which forces $$C_{t_1,\varepsilon,\kappa}=\varnothing.$$
Finally using again Proposition \ref{lem:nJumps}, but now for death events, ensures that
$$\P\left( D_{t_1,\varrho, \varepsilon}
 \right) \leq c_{\varepsilon,\varrho,K_0} \gamma_{\varepsilon, \varrho}^{K_0MN}.$$
 Recalling \eqref{ecureuil} and gathering the previous estimates for $A,B,C,D$ proves the expected result for $N_B(t_1)$ and ends the proof.
\end{proof}

\begin{Cor} \label{prop:HPbounds}
Let  $K_0\geq 1$. 
Then, there exists $c,d,C > 0$ and $\gamma \in (0,1)$  such that for all $M,N\geq 1$ and any initial condition $\mathbf U^{M,N}_0$,
$$ \P \left( \sup_{s\geq 0} \left\{\left(I_B^{M,N}(s)+ I_D^{M,N}(s)+ [\mathbf U^{M,N}(s)]_M \right)e^{-cs} \right\} \, \geq   d \right) \leq C\left( \P([\mathbf U^{M,N}_0]_M \geq  K_0) +\gamma^{MN}\right) .$$
Moreover, the constants $c,d,C, \gamma$ depend only on $\rho_0, \alpha, K_0$ .
\end{Cor}
\begin{proof}
We use $t_1>0$ and $\kappa>1$ coming from the previous proposition and consider the increments of $I_B$ and $I_D$  :
$$\Delta_k= I_B((k+1)t_1)-I_B(kt_1)+ I_D((k+1)t_1)  - I_D(kt_1).$$
The proof involves the first time when a deviation is observed. By Markov property at time $kt_1$ for $k\geq 0$, Proposition \ref{gdevIMU} ensures that  for any $k\geq 0$,
$$\P\left(\left\{\Delta_k\geq \kappa^{k+1} K_0  \right\} \cup
\left\{ \sup_{s\in [kt_1,(k+1)t_1]} [\mathbf U(s)]_M \geq \kappa^{k+1} K_0 \right\} \,  \bigg\vert \,  [\mathbf U(kt_1)]_M \leq \kappa^k K_0  \right) \lesssim \gamma^{\kappa^kK_0MN}.$$
As $\kappa>1$ and $\gamma <1$, the right hand is summable for $k\geq 0$ and its sum is dominated by $\gamma^{K_0MN}$. We get, by total probability formula,
\begin{align*}
 &   \P\left( \exists k\geq 0 \, \text{ such that } \,  \Delta_k\geq \kappa^{k+1} K_0  \,  \text{ or } \, 
\sup_{s\in [kt_1,(k+1)t_1]} [\mathbf U(s)]_M \geq \kappa^{k+1} K_0   \right) \\
& \qquad \leq \P([\mathbf U_0]_M \geq K_0)\\
&\qquad \quad + \sum_{k\geq 0} \P\left(\left\{\Delta_k\geq \kappa^{k+1} K_0  \right\} \cup
\left\{ \sup_{s\in [kt_1,(k+1)t_1]} [\mathbf U(s)]_M \geq \kappa^{k+1} K_0 \right\} \,  ; \,  [\mathbf U(kt_1)]_M \leq \kappa^k K_0  \right)\\
&\qquad  \lesssim \P([\mathbf U_0]_M \geq K_0)+ \gamma^{K_0MN}.
\end{align*}
Adding that for any $n\geq 0$ and  $t\in [nt_1,(n+1)t_1]$, by monotonicity of $I_B,I_D$,
$$(I_B(t)+I_D(t))e^{-ct} \leq \sum_{k=0}^n   \Delta_k e^{-ckt_1}, \qquad  \sup_{s\geq 0} [\mathbf U(s)]_Me^{-cs} \leq \max_{k=0,...,n} \sup_{s\in [kt_1,(k+1)t_1]} [\mathbf U(s)]_Me^{-ckt_1} $$
yields the result.
\end{proof}

\subsection{Quantitative estimates and proof of Theorem \ref{th:stoc_cv}}
\label{secpreuveth3}
 We can now quantify  the gap between the stochastic process and the crossed diffusion PDE :
$$\mathbf Z^{M,N} = \mathbf U^{M,N} -\widehat u^{M} , \qquad
\mathbf W^{M,N} =  \mathbf V^{M,N} - \widehat v^{M},$$
where $(u,v) \in L^\infty(Q_T) \cap L^2([0,T]; H^3(\T))$ is solution of \eqref{eq:SKT} satisfying the smallness condition \eqref{eq:petitesse2} and    $(\widehat u, \widehat v)$ is given by \eqref{eq:discSKTsource}.
The semi martingale decomposition of ${\bf U}^{M,N}$ in \eqref{decsemU}-\eqref{defphi1} can be written as follows 
\begin{align*}
 \mathbf U^{M,N}(t)&=\mathbf U^{M,N}_0+\int_0^t [\Delta_M(\mu_1(\mathbf U^{M,N}(s), \mathbf V^{M,N}(s) )\mathbf U^{M,N}(s))]\\ &\qquad \qquad \qquad \qquad \qquad +R_1(\mathbf U^{M,N}(s), \mathbf V^{M,N}(t)) \dd s
 + \mathcal M_1^{M,N}(t),   
\end{align*}
where we recall that $\mathcal M_1^{M,N}$ is a (square integrable) martingale given in  \eqref{expressM1}. 
We are now in position to apply  our stability result
(Proposition \ref{prop:disc_stab}) and use martingale inequalities and large deviations estimates of the previous subsection to evaluate the gaps $\mathbf Z^{M,N}$ and $\mathbf W^{M,N}$.
\begin{Prop}\label{prop:stabsoc}
Assume that Hypothesis \ref{hyp}, \ref{hyp:existSol} and  \ref{hyp:bd} hold and let $p \in [2, 2p_0/3)$. Then there exists  $c >0$ and  $(\delta_M')_M$ going to $0$ such that for any $T\geq 0$ and $M,N\geq 1$,
\begin{multline}
    \label{firstestim}\E[|||\mathbf Z^{M,N}|||_{T,M}^p + |||\mathbf W^{M,N}|||_{T,M}^p]  
     \lesssim e^{\exp(cT)} \E\left[\norm{ \mathbf Z_0^{M,N}}_{-1,M}^p + \norm{\mathbf W_0^{M,N} }_{-1,M}^p \right] \\
   + e^{\exp(cT)}\biggl(\delta_M' + \frac {1} {N^{\frac p 4}}\E[1 + |||\mathbf U^{M,N}|||_{T,M}^{p} + |||\mathbf V^{M,N}|||_{T,M}^{p}]\biggl). 
\end{multline}
Furthermore, 
    $$\E[|||\mathbf U^{M,N}|||_{T,M}^p + |||\mathbf V^{M,N}|||_{T,M}^p]
    \lesssim e^{\exp(cT)} \E\left[\norm{ \mathbf U_0^{M,N}}_{-1,M}^p + \norm{\mathbf V_0^{M,N} }_{-1,M}^p \right].$$
\end{Prop}
\begin{proof}
Proposition \ref{prop:disc_stab} applied to $\mathbf Z$ and $\mathbf U$ yields
\begin{align}
    |||\mathbf Z|||_{T,M}^2 + |||\mathbf W|||_{T,M}^2  
    &  \lesssim 
    e^{\Lambda(T)}  \Big(\norm{ \mathbf Z_0}_{-1,M}^2 + \norm{\mathbf W_0 }_{-1,M}^2 \label{Bornestabas} +\delta_M
    + \norm{H_1}_{L^\infty_T} + \norm{H_2}_{L^\infty_T}\Big),
    \end{align}
where 
$$ \Lambda(T) \lesssim \int_\QTM 1+ \mathbf U +\mathbf V +d_1( \mathbf U, \mathbf V)^2 + d_2( \mathbf U, \mathbf V)^2, $$
and $H_1$  (resp. $H_2$) have been defined 
in \eqref{eq:defEstoc}. Let us first control  $\Lambda(T)$ using Hypotheses  \ref{hyp:bd} $ii)$:
\begin{align*}
    \Lambda(T) 
    &\lesssim \int_\QTM 1+  \mathbf U(1+ d_1(\mathbf U,\mathbf V)) + \mathbf V (1+ d_2(\mathbf U,\mathbf V)) \\
    &\lesssim T + T \sup_{0 \le s \le T} [\mathbf U(s)]_M + T \sup_{0 \le s \le T} [\mathbf V(s)]_M +  \left(I_D^1(T)+ I_D^2(T)\right),
\end{align*}
where $$I_D^1(T)=\int_\QTM  \mathbf U d_1(\mathbf U,\mathbf V) \qquad \text{ and }  \qquad I_D^2(T)=\int_\QTM  \mathbf V d_2(\mathbf U,\mathbf V).$$
Let $K_0$ be given by Hypotheses \ref{hyp:bd} $i)$. We can invoke  Corollary \ref{prop:HPbounds} both for the processes  $({\bf U}, I_D^1)$ and  for $({\bf V}, I_D^2)$ to control the four terms on the right hand side with high probability. This proves the existence of $c,C>0$ and $\gamma\in (0,1)$ such that  for any $T\geq 0$ and $M,N\geq 1$,
$$\P(\Lambda(T) \geq C  \exp(cT)) \lesssim  \P([\mathbf U^{M,N}_0]_M \geq  K_0) + \gamma^{MN}.$$
Since $K_0$ is  constant and $\gamma^{MN} \lesssim (MN)^{-p_0} \lesssim M^{-p_0}$,  Hypotheses \ref{hyp:bd} $i)$  yields
$$\P(\Lambda(T) \geq C \exp(cT)) \lesssim  \frac {1}{M^{p_0}}.$$
Then, by Hölder's inequality, for all $q \ge 1$,
\begin{align*} 
\E\left[\1_{\Lambda(T) \geq C\exp(cT)}|||\mathbf Z|||_{T,M}^p\right] & \le \P\big( \,\Lambda(T) \geq C\exp(cT) \, \big)^{1/q'} \, \E[|||\mathbf Z|||_{T,M}^{qp}]^{1/q} \\
& \lesssim \frac {1}{M^{p_0/q'}} \, \left(|||\widehat u|||_{T,M}^{qp} + \E[|||\mathbf U|||_{T,M}^{qp}]\right)^{1/q}, \end{align*}
where $1/q'+1/q=1$.
The last term  of the right hand side can be bounded thanks to Corollary \ref{cor:BoundTM} and Proposition \ref{prop:mom_p} when $qp<p_0$.  Using also $\norm{\mathbf U(s)}_{1,M} = [\mathbf U(s)]_{M}$, we obtain
\begin{align*}
\E[\1_{\Lambda(T) \geq C\exp(cT)}|||\mathbf Z|||_{T,M}^p] &\lesssim \frac {1}{M^{p_0/q'}} \left(\norm{u}_{\infty}^{qp} +  (1+ TM)^{qp/2}\E[\sup_{s \le T}\norm{\mathbf U(s)}_{1,M}^{qp}]\right)^{1/q}\\
&\lesssim  M^{p/2- p_0/q'}  e^{cT}, \end{align*}
for  $c>0$ which allows to absorb the $T^{p/2}$. Since $p <  2 p_0/3$, we can now choose $q\in ( {p_0}/p, 1/(1 - p/ {2p_0}))$. Then $qp<p_0$ and  $p/2 < p_0/q'$, so that  the right hand side of the last estimate goes to zero.
This provides the expected estimate on the gap $\mathbf Z$ on the event $\{\Lambda(T)\geq C\exp(cT)\}$
and we know consider the complementary event.\\

We  elevate  the a.s. estimate \eqref{Bornestabas} to the power $p/2$. Taking then expectation and bounding $\E(\norm{H_i}_{L^\infty_T})$ thanks to Lemma \ref{lem:boundE1}, we obtain
\begin{multline*}
    \E\left[\1_{\Lambda(T)< C\exp(cT)}(|||\mathbf Z|||_{T,M}^p + |||\mathbf W|||_{T,M}^p )\right] \\
    \qquad \lesssim e^{\frac p 2 C\exp(cT)} \left(\E[\norm{ \mathbf Z_0}_{-1,M}^p + \norm{\mathbf W_0 }_{-1,M}^p ]+ \delta_M^{p/2} + \left(\frac {1+T^2} N \right)^{\frac p 4}\E[1 + |||\mathbf U|||_{T,M}^{p} + |||\mathbf V|||_{T,M}^{p}]\right) \! .
\end{multline*}
Gathering the bounds for the event $A$ and its complementary event and observing that $(1+T^2)^{\frac p 4} \lesssim \exp(\exp({cT}))$ ends the proof of \eqref{firstestim}. \\

We can  now apply  \eqref{firstestim} to compare $(\mathbf U, \mathbf V)$ the null solution $(u,v) = (0,0)$. In this case, $\mathbf r=0$, i.e. $\delta_M=0$,  and $(\mathbf Z, \mathbf W)=(\mathbf U, \mathbf V)$. We obtain
$$\E[|||\mathbf U|||_{T,M}^p + |||\mathbf V|||_{T,M}^p]
    \lesssim e^{\exp(cT)} \left(\E[\norm{ \mathbf U_0}_{-1,M}^p + \norm{\mathbf V_0 }_{-1,M}^p ] + \frac {1} {N^{\frac p 4}}\E[1 + |||\mathbf U|||_{T,M}^{p} + |||\mathbf V|||_{T,M}^{p}]\right) \! .$$
For $N$ large enough, the $|||\cdot|||_{T,M}$ norms from the right-hand side can be absorbed in the left-hand side. Recalling that Hypothesis $\ref{hyp:bd}$ $i)$ bounds the moments of the initial condition concludes the proof.
\end{proof}
We are now in position to prove the following approximation result, which implies Theorem \ref{th:stoc_cv} for $p=2$.
\begin{Thm} 
   Assume that Hypotheses \ref{hyp}, \ref{hyp:existSol} and  \ref{hyp:bd} hold and let $p \in [2, 2 p_0/3)$.
Then, there exists a sequence $(\delta_M)_M$ going to $0$ and a constant $c>0$ such that
the following estimate holds for any $T\geq 0$ and $N,M\geq 1$,
\begin{align*} 
    &\E \left[|||\pi_M(\mathbf U^{M,N}) - u|||_T^p +  |||\pi_M(\mathbf V^{M,N}) -v |||_T^p \right]  \lesssim e^{\exp(cT)}\left(D_{0,p}^{M,N} +\delta_M 
    +  \frac {1} {N^{p/4}}\right),
\end{align*}
where $\delta_M'$  goes to $0$ as $M$ tends to infinity and
\begin{equation*}
D_{0,p}^{M,N}=
\E\left[\norm{\pi_M(\mathbf U^{M,N}_0) - u_0 }_{H^{-1}(\T)}^p+ \norm{\pi_M(\mathbf V^{M,N}_0) - v_0 }_{H^{-1}(\T)}^p\right].
\end{equation*}
\end{Thm}
\begin{proof} 
Proposition \ref{prop:stabsoc} ensures that
\begin{align*}
   & \E[|||\mathbf Z|||_{T,M}^p + |||\mathbf W|||_{T,M}^p] \\
   &\qquad   \lesssim e^{\exp(cT)} \left(\E\left[\norm{ \mathbf Z_0}_{-1,M}^p + \norm{\mathbf W_0 }_{-1,M}^p \right]+ \delta_M'+ \frac {1} {N^{\frac p 4}}\E[1 + \norm{\mathbf U_0}_{-1,M}^{p} + \norm{\mathbf V_0}_{-1,M}^{p}]\right),
\end{align*}
which can be reduced using the inequality $ \norm{\cdot}_{-1,M} \lesssim \norm{\cdot}_{1,M}$ from Proposition \ref{prop:2bis} and the uniform bound on moments of the $\norm{\cdot}_{1,M}$ norm implied by Hypothesis \ref{hyp:bd} i) to
$$ \E[|||\mathbf Z|||_{T,M}^p + |||\mathbf W|||_{T,M}^p] \\
    \lesssim e^{\exp(cT)} \left(\E\left[\norm{ \mathbf Z_0}_{-1,M}^p + \norm{\mathbf W_0 }_{-1,M}^p \right]+ \delta_M'+ \frac {1} {N^{\frac p 4}}\right).$$
Then, Proposition \ref{prop:interpol} from the appendix gives 
$$ |||\pi_M(\mathbf  U) - u |||_T \lesssim |||\mathbf  Z|||_{T,M} + \delta_M'',$$
 $$ \norm{ \mathbf  Z_0}_{-1,M} \lesssim \norm{\pi_M(\mathbf  U_0) - u_0 }_{H^{-1}(\T)} + \delta_M'' ,$$
 with $\delta_M'' \underset{M \to \infty}{\longrightarrow} 0$ and similar inequalities for the second species. It concludes the proof.
\end{proof}


\section{Appendix}

\subsection{Elementary inequalities}

\begin{Prop} \label{prop:ineg_p}
Let $p \ge 1$. There exists a constant $C_p$ depending only on $p$ such that for all $x, \varepsilon \ge 0$ the following inequalities hold:\begin{enumerate}[label=(\roman*)]
    \item If $x \ge \varepsilon$, $x^p - (x- \varepsilon)^p \le C_p \varepsilon x^{p-1}$,
    \item $(x+\varepsilon)^p - x^p \le C_p \varepsilon x^{p-1} + C_p \varepsilon^p$,
    \item If $x \ge \varepsilon$ and moreover $p\ge 2$, then $(x+ \varepsilon)^p - x^p \le x^p - (x- \varepsilon)^p + C_p \varepsilon^2 x^{p-2}$.
\end{enumerate}
\end{Prop}
\begin{proof}
For the first inequality, write
$$x^p - (x- \varepsilon)^p = p \int_{x- \varepsilon}^x y^{p-1} \, \dd y \le  p \int_{x- \varepsilon}^x x^{p-1} \, \dd y \le p \varepsilon x^{p-1} .$$
The second one can then be deduced from the first one.
For the third, if $p \ge 2$ then using the first inequality
\begin{align*}
    (x+ \varepsilon)^p - x^p - (x^p - (x- \varepsilon)^p) &= p \int_{x- \varepsilon}^x \left((y+ \varepsilon)^{p-1} -  y^{p-1} \right) \, \dd y \\
    & \le p\int_{x- \varepsilon}^x (p-1) \varepsilon (y+ \varepsilon)^{p-2}  \, \dd y \\
    & \le p(p-1) \varepsilon  \int_{x- \varepsilon}^x  (x+ \varepsilon)^{p-2} \, \dd y 
    \le p(p-1) 2 ^{p-2} \varepsilon^2  x^{p-2},
\end{align*}
which ends the proof in the case $p \ge 2$. If $p \in [1, 2)$, the proof is similar: instead of $y+ \varepsilon)^{p-1} -  y^{p-1} \le (p-1) \varepsilon (y+ \varepsilon)^{p-2} $ we have $y+ \varepsilon)^{p-1} -  y^{p-1} \le (1-p) \varepsilon y^{p-2} $, which can then be dragged through the previous successive inequalities to get
$$ (x+ \varepsilon)^p - x^p - (x^p - (x- \varepsilon)^p) \le C_p \varepsilon^2  (x- \varepsilon)^{p-2}.$$
If $x \ge 2 \varepsilon$, this implies
$$ (x+ \varepsilon)^p - x^p - (x^p - (x- \varepsilon)^p) \le C_p \varepsilon^2  \left(x- \frac x 2 \right)^{p-2} = C_p \varepsilon^2  x^{p-2}.$$
If $\varepsilon \le x \le 2 \varepsilon$, we actually have
$$(x+ \varepsilon)^p - x^p - (x^p - (x- \varepsilon)^p) \le (x+ \varepsilon)^p + (x- \varepsilon)^p \le (3\varepsilon)^p + \varepsilon^p \le (3^p +1)\varepsilon^2 x^{p-2},$$
which concludes the proof.
\end{proof}

\subsection{Sobolev spaces and Poisson equation on the torus} \label{sec:sobolev}
We collect in this short paragraph classical results about Sobolev spaces and the Poisson equation considered on the torus. Recall that any periodic distribution $f\in\mathscr{D}'(\T^d)$ decomposes
\begin{align*}
    f = \sum_{k\in\mathbb{Z}^d} c_k(f) e_k,
\end{align*}
where $e_k:x\mapsto e^{2i\pi k\cdot x }$ and $c_k(f) := \langle f,e_k\rangle$.
\begin{Def}
For $s\in\R$, we define $H^s(\T^d)$ as the set of periodic distributions $f$ such that 
\begin{align*}
    \|f\|_{H^s(\T^d)}:= \left(\sum_{k\in\mathbb{Z}^d} |c_k(f)|^2 (1+|k|^2)^s\right)^{1/2} < +\infty.
\end{align*}
Equipped with $\|\cdot\|_{H^s(\T^d)}$, $H^s(\T^d)$ is a Hilbert space. The closed subspace of mean-free elements of $H^s(\T^d)$ is the kernel of the linear form $f\mapsto c_0(f):=[f]$ and is denoted $H^s_m(\T^d)$.
\end{Def}
Note that $\norm{ \cdot}_{H^0} = \norm{ \cdot}_{L^2}$ and $\norm{ \cdot}_{H^s} \le \norm{ \cdot}_{H^r}$ for $s \le r$. In particular, $\norm{ \cdot}_{H^{-1}} \le \norm{ \cdot}_{L^2}$. 
With this functional setting at hand we recall the following standard facts for Poisson's equation:
\begin{Prop} \label{prop:elliptique}
    Let $f \in H_m^{-1}(\T^d)$ a mean-free function. Then, there exists a unique mean-free function $\phi \in H^{1}_m(\T^d)$, also denoted $\Delta^{-1} f$, solution of $\Delta \phi = f$. Moreover, $\norm {\nabla \phi}_{L^2(\T^d)}=\norm f _{H^{-1}(\T^d)}$. If $f \in H^k_m(\T^d)$ for a $k \in \mathbb N$, then $\Delta^{-1} f \in   H^{k+2}_m(\T^d)$ and the application $f \mapsto \Delta^{-1} f$ is continuous $H^k_m(\T^d) \longrightarrow H^{k+2}_m(\T^d)$. In particular, if $f$ is smooth then so is $\Delta^{-1} f$.
\end{Prop}
From this proposition we see that $\nabla \Delta^{-1}$ induces an isometry from $H^{-1}_m(\T^d)$ to $L^2(\T^d)$, so that using a straightforward integration by parts we have the following expression for the $H^{-1}(\T^d)$ inner-product  of two elements $f$ and $g$ of $H^{-1}(\T^d)_m$ 
$$ \langle f,g \rangle_{H^{-1}} = \langle \nabla \Delta^{-1} f,\nabla \Delta^{-1} g\rangle_{L^2} = -\langle f,\Delta^{-1} g\rangle_{L^2}. $$
Of course, this formula extends to the whole $H^{-1}(\T^d)$ after adding the constant components so that for any $f,g\in H^{-1}(\T^d)$ there holds,
$$\langle f ,g \rangle_{H^{-1}(\T^d)} =  [f][g] - \langle  f - [f],  \Delta^{-1} (g - [g]) \rangle_{L^2}.$$
This last expression of the $H^{-1}(\T^d)$ inner-product  motivates the following definition of inner-product on $\ell^2(\T^d_M)$
$$\langle \mathbf u , \mathbf v \rangle_{-1,M} = [\mathbf u]_M [\mathbf v]_M - \langle \mathbf u -[\mathbf u]_M , \Delta_M^{-1}(\mathbf v -[\mathbf v]_M) \rangle_{-1,M}.$$
Lastly, for $d=1$, defining for $h\neq 0$ and $f\in\mathscr{D}'(\T)$ the operator $\Delta_h f:= h^{-2}(\tau_h f + \tau_{-h} f - 2f)$, we have (see for instance \cite[Subsection 6.2]{mbh}) that: 
\begin{Prop}\label{prop:FD}
    For $f \in H^2(\T)$, we have$\norm{ \Delta_h f }_{L^2(\T)} \le \norm{  f }_{H^2(\T)}$ and in $L^2(\T)$,
    $$ \Delta_h f \underset{h\to 0}{\longrightarrow} \Delta f.$$
\end{Prop}

\subsection{A non-linear Grönwall lemma}

\begin{Lemme}\label{lem:Gronwall}
Let $\psi, \phi: \R_+ \rightarrow \R_+$ be continuous by parts, $K_0 \ge 0$, $a \in L^\infty_{loc}(\R_+)$ et $b \in L^1_{loc}(\R_+)$:
 $$  \psi(t) + \phi(t) \le K_0 + \int_0^t [a(s)\psi(s) + b(s) \psi(s)^{1/2}  ]  \, \dd s $$
 Then:
 $$\sup_{s \le t} \psi(s)+\phi(t) \le 2 e^{\norm {a}_{L^1([0,t])} } \left(K_0 + \norm b _{L^1([0,t])}^2 \right).$$
\end{Lemme}
\begin{proof}

Set $f(t) = K_0 + \int_0^t [a(s)\psi(s) + b(s) \psi(s)^{1/2} ] ds $. Then $\psi(t) \le f(t)$, so $f'(t) \le a(t) f(t) + b(t) f(t)^{1/2}$.
Dividing on both sides by $y(t):=f(t)^{1/2}$,
$$ 2y'(t) \le a(t) y(t) + b(t).$$
The classical Grönwall lemma then gives
$$y(t) \le y(0)e^{\frac 1 2 \int^t_0 a} + \frac 1 2 \int_0^t  e^{\frac 1 2 \int_s^t a} b(s) \, \dd s,$$
$$(\psi(t)+\phi(t))^{1/2} \le f(t)^{1/2} \le y(t) \le e^{\frac 1 2 \norm {a}_{L^1([0,t])} } \left(K_0^{1/2} + \int_0^t b(s) \, \dd s\right).$$
Squaring this inequality,
$$\psi(t)+\phi(t) \le 2 e^{\norm {a}_{L^1([0,t])} } \left(K_0 + \norm {b}_{L^1([0,t])}^2\right),$$
hence the result by taking the supremum.
\end{proof}

\subsection{Interpolation operators}

The goal of this appendix is to prove the following lemma, which is used to translate estimates on the discrete problem to estimates in the continuous space at the end of the proofs of Theorems \ref{th:disc_stab} and \ref{th:stoc_cv}:

\begin{Prop}\label{prop:interpol}
    Let $u \in L^2([0,T];H^3(\T))$ a solution of the SKT system \eqref{eq:SKT}, $\mathbf u^M \in \mathcal C ([0,T], \ell^2(\T_M))$ a sequence of functions such that $\sup{M \in \mathbb N} \norm{\mathbf u^M(0)}_{2,M} < \infty $. Setting $\zeta^M = \pi_M(\mathbf u^M) - u$ and $\mathbf z^M = \mathbf u^M - \widehat u^M$, there exists a sequence $\delta_M$ depending only on $u$ and $\sup{M \in \mathbb N} \norm{\mathbf u^M(0)}_{2,M} $ such that $\delta_M \underset{M \to \infty}{\longrightarrow} 0$ and
    $$|||\zeta ^M|||_T \lesssim |||\mathbf z^M|||_{T,M} + \delta_M.$$
    $$\norm{\mathbf z^M(0)}_{-1,M} \lesssim \norm{\zeta^M(0)}_{H^{-1}(\T)} + \delta_M$$
\end{Prop}

To do so, we will rely on estimates for the operator $\pi_M$, which has already been studied in \cite{mbh}, Propositions 2 and 5:

\begin{Prop} \label{prop:norm_equiv}
For all $u \in \ell^2(\T_M)$,
    $$\norm{ \pi_M (u) }_{L^2(\T)} \le \norm{ u }_{2,M}. $$
Moreover, we have a uniform equivalence
    $$ M \norm{\pi_M(u)}_{H^{-1}(\T)} +  \norm{\pi_M(u)}_{L^2(\T)} \sim M \norm{u}_{-1,M} +  \norm{\pi_M(u)}_{L^2(\T)}.$$
\end{Prop}

More specifically, we will use the following consequence of these inequalities:

\begin{Cor} \label{cor:12} For all $u \in L^2([0,T] \times \T_M)$,
$$|||\pi_M (u) |||_{T} \lesssim ||| u |||_{T,M}.$$
\end{Cor}
\begin{proof}
The uniform equivalence gives, for all $v \in  \ell^2(\T_M)$,
$$\norm{\pi_M( v)}_{H^{-1}(\T)} \lesssim  \norm{v}_{-1,M} +  M^{-1} \norm{\pi_M(v)}_{L^2(\T)}.$$
We then use $\norm{\pi_M(v)}_{L^2(\T)} \le \norm{v}_{2,M} \le M \norm{v}_{-1,M}$, which gives $\norm{\pi_M( v)}_{H^{-1}(\T)} \lesssim \norm{v}_{-1,M}$. \newline
Therefore,
$$\sup_{s \le T} \norm{\pi_M( u(s))}^2_{H^{-1}(\T)} \lesssim \sup_{s \le T} \norm{u(s)}^2_{-1,M} \le  ||| u ||| ^2_{T,M}.$$
For the $L^2$ part,
$$ \int_0^T \norm{ \pi_M (u(s)) }_{L^2(\T)}^2 \, \dd s \le  \int_0^T \norm{ u(s) }_{2,M}^2 \, \dd s \le  ||| u |||^2_{T,M}, $$
hence the result.
\end{proof} 

\begin{proof}[Proof of Proposition \ref{prop:interpol}]
We decompose
$$\zeta^M  = \pi_M(\mathbf u^M - \widehat u^M) + \pi_M(\widehat u^M) - u = \pi_M(\mathbf z^M) + \pi_M(\widehat u^M) - u.$$
Corollary \ref{cor:12} yields the following upperbound:
$$|||\zeta^M |||_T \lesssim |||\mathbf z^M|||_{T,M} + |||\pi_M(\widehat u^M) - u|||_{T,M}.$$
To prove the first part of the proposition, there remains to see that the last term converges to $0$. This is ensured by Lemma 6.2.10 of \cite{allaire}:
\begin{Lemme}\label{lem:0022}
Let $f\in H^2(\T)$. The two following inequalities hold:
   $$\norm{ f - \pi_M(\widehat f^M)}_{L^2(\T)} \lesssim M^{-2} \norm f _ {H^2(\T)},$$
   $$\norm{ f - \pi_M(\widehat f^M)}_{H^{-1}(\T)} \lesssim M^{-2} \norm f _ {H^2(\T)}.$$
\end{Lemme}
This lemma implies $|||\pi_M(\widehat u^M) - u|||_{T,M} \lesssim M^{-2} \norm u _ {L^\infty_T H^2}$. Adding that $u \in L^\infty([0, T ]; H^2(\T)) $, this end the proof of the first part.

For the second part, by Proposition \ref{prop:norm_equiv}
\begin{align*}\norm{ \mathbf z^M(0)}_{-1,M} &\lesssim \norm{ \pi_M(\mathbf z^M(0))}_{H^{-1}(\T)} + M^{-1} \norm{ \pi_M(\mathbf z^M(0))}_{L^2(\T)} \\
&\le \norm{\zeta^M(0)}_{H^{-1}(\T)}+ \norm{\pi_M(\widehat u_0^M) - u_0}_{H^{-1}(\T)} + M^{-1} \norm{\mathbf z^M(0)}_{2,M}.
\end{align*} 
Moreover, Lemma \ref{lem:0022} ensures that $\norm{\pi_M(\widehat u_0^M) - u_0}_{H^{-1}(\T)}$ goes to $0$. Bounding $\norm{\mathbf z^M(0)}_{2,M}$ by $\norm{u}_\infty + \sup_{M \in \mathbb N} \norm{\mathbf u^M(0)}_{2,M} $  ends the proof.
\end{proof}






\subsection{A discrete duality lemma}
The following duality lemma is used for the proof of Lemma \ref{lem:duality}  in Section \ref{sec:2}. It is  proved in \cite{mbh}, Lemma 4.14, writing $x=x_S'$:
\begin{Lemme} \label{lem:discont}
Let $\mu,f,x \in L^\infty ([0,T] \times \T_M)$ with $\mu$ uniformly lower bounded in space and time by some constant $\alpha >0$, and $x_J(t) = \sum_{k=1}^m \1_{t \ge t_k} a_k$ piecewise constant function on $[0,T]$. Let ${\bf z}_0: \T_M \rightarrow \R$. Assume that there is a $C^1$ piecewise function  ${\bf z}$  taking values in values in $\R^M$ and  solution for $t\geq 0$ of
$${\bf z}(t) = {\bf z}_0 + \int_0^t \left( \Delta_M[{\bf z}(s)\mu(s) + f(s)] +x(s)\right) \, \dd s +x_J(t).$$
Then, for any $T\geq 0,$ 
\begin{multline*}
    \norm {{\bf z}(t)} ^2_{-1,M} + \int_\QTM \mu {\bf z}^2
    \le \norm {{\bf z}_0} ^2_{-1,M} + \int_0^t [{\bf z}(s)]^2[\mu(s)] \,\dd s
    +\frac 1 \alpha \int_\QTM f^2 \\
    + 2 \int^t_0 \langle {\bf z}(s) ,x(s)\rangle_{-1,M} \, \dd s + \sum_{k=1}^m \norm { a _k}_{-1,M}^2 + 2\sum_{k=1}^m \langle {\bf z}(t_k^- ) , a _k\rangle_{-1,M}.
\end{multline*}
\end{Lemme}

\noindent {\bf Acknowledgement.} The authors would like to warmly thank Arnaud Debussche for suggestions who have led to improve the regime of convergence of the sequence of stochastic processes.  This work  was partially funded by the Chair “Mod\'elisation Math\'ematique et Biodiversit\'e" of VEOLIA-Ecole Polytechnique-MNHN-F.X., by the European Union (ERC, SINGER, 101054787), by the Fondation Mathématique Jacques Hadamard and by the ANR project NOLO (ANR-20-CE40- 0015), funded by the French Ministry of Research. Views and opinions expressed are however those of the author(s) only and do not necessarily reflect those of the European Union or the European Research Council. Neither the European Union nor the granting authority can be held responsible for them. 
\vspace{1mm}

\printbibliography[
]
\noindent\textsc{Vincent Bansaye\\
CMAP, INRIA, École polytechnique, Institut Polytechnique de Paris, 91120 Palaiseau, France.\\
E-mail address:} \href{mailto:vincent.bansaye@polytechnique.edu}{vincent.bansaye@polytechnique.edu}\\

\noindent\textsc{Alexandre Bertolino\\
Sorbonne Université, CNRS, Université de Paris, Inria, Laboratoire Jacques-Louis Lions (LJLL), F-75005 Paris, France and 
CMAP,  École polytechnique, Institut Polytechnique de Paris, 91120 Palaiseau, France.\\
E-mail address:} \href{mailto:alexandre.bertolino@ens.psl.eu}{alexandre.bertolino@ens.psl.eu}\\

\noindent\textsc{Ayman Moussa\\
Sorbonne Université, CNRS, Université de Paris, Inria, Laboratoire Jacques-Louis Lions (LJLL), F-75005 Paris, France.\\
E-mail address:} \href{mailto:ayman.moussa@sorbonne-universite.fr}{ayman.moussa@sorbonne-universite.fr}

\end{document}